\documentclass{sn-jnl}
\usepackage{enumerate}
\usepackage{amsthm}

\usepackage{graphicx}%
\usepackage{multirow}%
\usepackage{amsmath,amssymb,amsfonts}%
\usepackage{amsthm}%
\usepackage{mathrsfs}%
\usepackage[title]{appendix}%
\usepackage{xcolor}%
\usepackage{textcomp}%
\usepackage{manyfoot}%
\usepackage{booktabs}%
\usepackage{algorithm}%
\usepackage{algorithmicx}%
\usepackage{algpseudocode}%
\usepackage{listings}%
\usepackage{caption}
\usepackage{float} 

\numberwithin{equation}{section}

\usepackage{url}
\usepackage{hyperref}
\hypersetup{colorlinks=false, linkbordercolor={1 1 1}, citebordercolor={1 1 1}, urlbordercolor={1 1 1}} 
\usepackage[normalem]{ulem}
\usepackage{accents}

\newtheorem{theorem}{Theorem}[section]
\newtheorem{lemma}[theorem]{Lemma}

\newtheorem{remark}{Remark}
\newtheorem{proposition}[theorem]{Proposition}
\newtheorem{definition}{Definition}[section]
\newtheorem*{remark*}{Remark}

\newcommand{\R}{\mathbb R}
\newcommand{\Z}{\mathbb Z}
\newcommand{\N}{\mathbb N}
\newcommand{\EE}{\mathbb E}
\newcommand{\PP}{\mathbb P}
\newcommand{\VV}{\mathbb V}
\newcommand{\oVV}{{\overline{\mathbb V}}}
\newcommand{\pVV}{\partial\mathbb V}
\newcommand{\pE}{\partial E}
\newcommand{\oE}{{\overline{E}}}

\newcommand{\cB}{\mathcal B}

\newcommand{\cE}{\mathcal E}

\newcommand{\cG}{\mathcal G}

\newcommand{\cK}{\mathcal K}

\newcommand{\cM}{\mathcal M}
\newcommand{\cN}{\mathcal N}

\newcommand{\cP}{\mathcal P}

\newcommand{\cV}{\mathcal V}
\newcommand{\cU}{\mathcal U}

\newcommand{\cY}{\mathcal Y}

\newcommand{\tprod}{{\textstyle \prod}}

\def\sqr{\vcenter{
         \hrule height.1mm
         \hbox{\vrule width.1mm height2.2mm\kern2.18mm\vrule width.1mm}
         \hrule height.1mm}}                  
\def\square{\ifmmode\sqr\else{$\sqr$\vskip 3mm}\fi}

\newcommand{\one}{{\bf 1}\hskip-.5mm} 

\newcommand{\uk}{\underline k}

\newcommand{\oV}{\overline V}

\newcommand{\us}{\underline s}
\newcommand{\ux}{\underline x}

\newcommand{\ttau}{\tilde \tau}

\newcommand{\tPP}{\widetilde\PP}

\newcommand{\tO}{\tilde O}

\renewcommand{\paragraph}[1]{\vspace{1mm}\noindent{\bf {#1} \hspace{1mm}}}

\newcommand{\nn}{\nonumber}

\DeclareMathOperator{\stat}{stat}

\definecolor{cmm}{rgb}{0,.6,0.4}
\definecolor{cmm'}{rgb}{.6,0,.4}

\allowdisplaybreaks
\begin{document}

\centerline{\Large Hidden temperature in the KMP model}
\vspace {4mm}
\centerline{\large
  Anna de Masi
  \footnote{Gran Sasso Science Institute,
    {\href{anna.demasi@gmail.com}{\tt anna.demasi@gmail.com}},
    {\href{https://orcid.org/0000-0002-8154-2498}{orcid: 0000-0002-8154-2498}}}, 
  Pablo A. Ferrari
  \footnote{Universidad de Buenos Aires, corresponding author,
    {\href{pagfrr@gmail.com}{\tt pagfrr@gmail.com}},
    {\href{https://orcid.org/0000-0001-5395-6100}{orcid: 0000-0001-5395-6100}}},
  Davide Gabrielli
  \footnote{Università de L'Aquila,
    \href{davide.gabrielli@univaq.com}{\tt dvd.gabrielli@gmail.com},
  {\href{https://orcid.org/0000-0001-8776-6081}{orcid: 0000-0001-8776-6081}}
}
}
\vspace {4mm}







{
{\bf Abstract. } In the Kipnis Marchioro Presutti (KMP) model a positive energy $\zeta_i$ is associated with each vertex $i$ of a finite graph with a boundary. When a Poisson clock rings at an edge $ij$ with energies $\zeta_i,\zeta_j$, those values are substituted by $U(\zeta_i+\zeta_j)$ and $(1-U)(\zeta_i+\zeta_j)$, respectively, where $U$ is a uniform random variable in $(0,1)$. A value $T_j\ge0$ is fixed at each boundary vertex $j$. The dynamics is defined in such way that the resulting Markov process $\zeta(t)$, satisfies that $\zeta_j(t)$ is exponential with mean $T_j$, for each boundary vertex $j$, for all $t$.  We show that the invariant measure is the distribution of a vector $\zeta$ with coordinates $\zeta_i=T_iX_i$, where $X_i$ are iid exponential$(1)$ random variables, the law of $T$ is the invariant measure for an opinion random averaging/gossip model with the same boundary conditions of $\zeta$, and the vectors $X$ and $T$ are independent. The result confirms a conjecture based on the large deviations of the model. When the graph is one-dimensional, we bound the correlations of the invariant measure and perform the hydrostatic limit. We show that the empirical measure of a configuration chosen with the invariant measure converges to the linear interpolation of the boundary values.}




\section{Introduction}

A major issue in non equilibrium statistical mechanics is the understanding of stationary non equilibrium states (SNS). Effective and interesting toy models of real systems are obtained considering multicomponent interacting stochastic systems. The irreversibility that generates the non equilibrium situation is created by putting the system in contact with external boundary sources with different features.

In the non equilibrium case we have a flow of mass, heat, charges and the classic well established Gibbsian equilibrium framework does not apply. In particular, typically long range correlations appear.

A SNS for an interacting stochastic system is therefore the invariant measure of a Markovian stochastic process having as state space a configuration of particles or energies located on the vertices of a typically large graph. The process is not reversible so that we obtain a model for non equilibrium.

The thermodynamic behavior of the system is obtained by considering a very large number of components and a scaling limit. The limiting behavior is given by macroscopic variables, the so called empirical measures. Then one studies fluctuations in the framework of large deviations.

An exact form of the large deviations rate functional for the empirical measure of the SNS has been obtained for only a few one dimensional boundary driven models. See for example \cite{MR2335699,MFT} for general reviews. The rate functional is not local due to the presence of long range correlations and it is obtained by the solution of a variational problem. Among the solvable models we have the exclusion process and the Kipnis-Marchioro-Presutti (KMP) process \cite{KMP,BGL}.

In the present paper we study the KMP process and show that the complex statistical structure of the stationary state is better understood introducing hidden variables. The KMP process describes the stochastic evolution of the energy associated to harmonic oscillators organized along a chain. Each oscillator $i$ is characterized by a positive energy $\zeta_i$ and nearest neighbors oscillators exchange their energies at random exponential times. When two oscillators exchange energy, they simply share uniformly the total energy of the two. Each boundary vertex $j$ updates its value to an exponential random variable of mean $T_j$, the boundary condition.

{
\centering
\includegraphics[width=.7\textwidth]{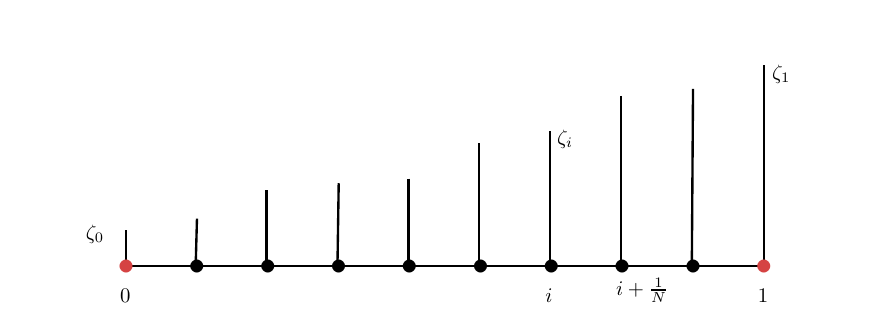}
\captionof{figure}{A one dimensional graph with vertex set $\{0,\frac1N,\dots,\frac{N-1}N,1\}$, boundary $\{0,1\}$ and nearest neighbor edges, and a KMP configuration $\zeta$. The boundary values $\zeta_0,\zeta_1$ are exponentially distributed with means $T_0,T_1$, respectively.\label{graph1} }\vspace{1mm}
}

It is instructive to recall the large deviations rate functional of the rescaled one dimensional boundary driven KMP model \cite{BGL}.
Consider the system in the interval $[0,1]$ and call $\rho(x)$ the macroscopic density of energy at the point $x\in [0,1]$. The variables $\zeta_i$ are associated to a lattice with $N$ points and mesh $1/N$. The two external sources at temperatures respectively $T_{\pm}$ are located at the right and left boundary point of the interval. Without loss of generality we assume $T_0=T_-<T_+=T_1$. See Fig.\/\ref{graph1}. We have the following large deviation principle:
\begin{equation}
\mathbb P_{\textrm{stat}}\left(\pi_N(\zeta)\sim \rho(x)dx\right)\simeq e^{-NV(\rho)}\,,
\end{equation}
where $\mathbb P_{\textrm{stat}}$ denotes the law of the stationary state, $\pi_N$ is the empirical measure (i.e. the macroscopic variable associated to the microscopic configuration of energies $\zeta$) and $V$ is the large deviations rate functional.

The variational expression of $V$ is given by
\begin{equation}\label{quasipot-kmp}
V(\rho)= \inf_{\tau}\int_0^1\left[\frac{\rho(x)}{\tau(x)}-1 -\log\frac{\rho(x)}{\tau(x)}-\log \frac{\tau'(x)}{T_+-T_-}\right] dx\,,
\end{equation}
where the infimum is over monotone increasing differentiable functions $\tau$ that satisfy the boundary conditions $\tau(0)=T_-$ and $\tau(1)=T_+$.

The origin of non locality of the rate functional and therefore of the long range correlations of the SNS is generated by the minimization over the function $\tau$. The minimal $\tau$ solves the Euler Lagrange equation that is a non linear differential equation depending on $\rho$.

Formula \eqref{quasipot-kmp} suggests that the statistical properties of the SNS may be better understood introducing and observing new hidden random variables $T$ whose macroscopic observable is  $\tau$. Indeed the joint functional on $(\rho,\tau)$ on the right hand side of \eqref{quasipot-kmp} may be interpreted as a rate functional of a joint distribution of variables $(\zeta, T)$ and the minimization procedure becomes then naturally the contraction principle in action. The contraction principle gives the rate functional on one single observable minimizing over the possible values of the other one in the joint rate functional.

This possible interpretation of \eqref{quasipot-kmp} has been proposed in \cite{Bertini_2007}. Indeed  in \cite{Bertini_2007} it has been observed that the joint functional on the right hand side of \eqref{quasipot-kmp} may be interpreted as the rate functional for the macroscopic observables associated to the random variables $(\zeta, T)$ when $T$ are the order statistics of uniform random variables in $[T_-,T_+]$ and, conditioned to the values of the variables $T$, the variables $\zeta$ are independent exponential random variables with average respectively $T$. This simple measure is however not the distribution of the SNS: an explicit computation was possible only in the case of one single oscillator giving as law of the single $T$ variable the arc sine law.

In this paper we partially answer the unsolved issues in \cite{Bertini_2007}. We consider a finite graph, and introduce an interacting stochastic dynamic of the variables $T_i$ coupled to the energies $\zeta_i$ of the boundary driven KMP process, $i$ denotes a vertex of the graph. The variable $T_i$ has a natural interpretation as the temperature associated to the oscillator $i$. When two oscillators interact exchanging their energies $\zeta$, at the same time the two oscillators thermalise and the variables $T$ describing the temperatures assume the same value, suitably distributed between the two values before thermalization. The evolution of the joint process $(T(t),\zeta(t))$ is described by introducing a KMP process  $X(t)$, with identical boundary conditions $T_j=1$, for each boundary vertex $j$. If $X(0)$ is distributed with the equilibrium invariant measure and $X(0)$ is independent of $T(0)$, then $X(t)$ and $T(t)$ are independent for all $t$, and the marginal evolution of $T(t)$ coincides with the evolution of an auxiliary Markov 
process $O(t)$. This is the content of Theorem \ref{T2}  which is one of our main results.
We then define $\zeta_i(t) = O_i(t)\,X_i(t)$, for all vertices $i$ of the graph. 
 We prove that  the invariant measure of the boundary driven KMP process $\zeta(t)$ is a mixture of exponential distributions whose average values are distributed according to the unique invariant measure of the $O(t)$ process; this is the content of Theorem \ref{teo3}, our main result. 
 
Thus in order to characterize the invariant measure of the boundary driven KMP process we have to study the invariant measure of the $O$ process.
 
The process $O$ fits in the class of opinion models: $O_i(t)$ is the opinion of the individual at the site $i$ at time $t$; boundary individuals have
fixed opinions,  $O_j(t)\equiv T_j$ for all $t$, if $j$ is in the boundary.  At unitary rate pairs of neighboring individuals commit to assume a common opinion, uniformly distributed between the two opinions before the interaction. Internal individuals  adopt the committed opinion, but boundary individuals keep their value disregarding the commitment.

 In the case of one single individual, the invariant measure of the opinion model can be exactly computed and coincides with the arc sine distribution.
 This gives a dynamic perspective to the computations in \cite{Bertini_2007}.

When there are $N$ oscillators in the segment $[0,1]$ with nearest neighbor interactions, we show in Proposition \ref{prop2.7} that the correlations behave as $1/N^2$ and show  in Theorem \ref{ihydr} that the hydrostatic limit for the law of the empirical measure of the opinion model is the delta measure concentrated on the linear interpolation of the boundary temperatures.

Using a coupling between a continuous and a discrete KMP process we extend the results to the discrete setting.  

An interesting problem is to investigate if an enlargement of the state space  is useful and interesting also in other solvable SNS like for example the boundary driven exclusion process. In this case there is a formula similar to \eqref{quasipot-kmp} but the variational problem involves a supremum and not an infimum. Consequently there is not a direct interpretation in terms of the contraction principle. We expect however that the same functional could be written equivalently as an infimum with a possible interpretation as a contraction principle. Transformations of this type occur in \cite{Bertini_2007,EnauD} and are suggested by the results in \cite{DEL}. 
A result of this type has been obtained for generalized zero range dynamics in \cite{Har,modena}, and for symmetric exclusion process in \cite{fgc2023}.

In Section \ref{results} we give the definitions and state the main results and at the end of the Section we give an outline of the paper.

\section{Definitions and results}
\label{results}

We consider an oriented graph $(\oVV , \oE)$, where $\oVV$ is a finite set of vertices, $\oE\subset\{ij:i,j\in \oVV\}$ is the set of oriented edges. The set of vertices is partitioned in internal and boundary, denoted by $\VV$ and $\pVV$; we have $\oVV=\VV\cup\pVV$. The set of edges with both extremes in $\VV$ are called internal, and denoted $E$. The set of vertices with one extreme in $\pVV$ is denoted $\pE$. There are no edges with both extremes in the boundary. We have $\oE = E \cup \pE$. Furthermore for each pair of vertices $i,j\in \oVV$ there is at most one edge between them, so that if $ij\in \oE$ then $ji\not\in \oE$. Boundary edges are always oriented towards the boundary vertex: if $ij\in \pE$ then $j\in \partial \VV$ and $i\in \VV$. A natural case is when the set of vertices is totally ordered and then the choice among the two possible orientation of an edge is determined by the order so that
$ij\in \oE$ implies $i<j$. In this case we may label the vertex at the origin by $0'$ and use the convention $0'>1$.

Each boundary vertex $j\in\pVV$ has an associated value $T_j$ a  fixed positive real number which we call temperature,  we denote by $T_{\pVV}=\{T_j,j\in \pVV\}$.

The boundary driven  Kipnis Marchioro Presutti model (KMP) \cite{KMP} is a Markov process $\zeta(t)=\{\zeta_i(t), i\in {\oVV}\}\in \mathbb R_+^{\oVV}$ with boundary temperatures $T_{\pVV}\in \mathbb R_+^{\pVV}$. To describe it, we use the Harris graphical construction. Associate a rate one marked Poisson process to each edge $ij\in \oE$. Call ``clock rings'' the events of the Poisson processes. The marks are iid Uniform random variables in $[0,1]$, and iid exponential random variables of rate 1; the exponential random variables are associated only to boundary edges.  When the clock rings at edge $ij$, update the current energy configuration $\zeta$ to an updated configuration $\zeta'$, where $\zeta'_\ell=\zeta_\ell$ for any $\ell\neq i,j$, and, calling $U$ the associated uniform mark, set $\zeta'_i=U(\zeta_i+\zeta_j)$, because $i\in\VV$ always, and set $\zeta'_j=(1-U)(\zeta_i+\zeta_j)$, if $j\in\VV$,  otherwise (if $j\in\pVV$), set $\zeta'_j= T_jB$, where $B$ is the associated exponential mark, so that $T_jB$ is exponential with mean $T_j$. The generator of the KMP process with boundary condition $T_{\pVV}$ is
\begin{align}
    L^\zeta f(\zeta)&:=\sum_{ij\in E}
                      \int_0^1du\,\bigl[f(H^\zeta_{ij;u}\zeta)-f(\zeta)\bigr]\label{Lzeta}\\
  &\qquad\qquad+\sum_{ij\in \partial E} \int_0^1du\int_0^{+\infty}db\,e^{-b}\,\bigl[f(H^\zeta_{ij;u,b}\zeta)-f(\zeta)\bigr],
\end{align}
where
\begin{align}
  (H^\zeta_{ij;u}\zeta)_{\ell}:=  \begin{cases}
    u(\zeta_i+\zeta_j)&\ell=i,\\
    (1-u)(\zeta_i+\zeta_j)&\ell=j,\\
    \zeta_\ell&\text{else},
  \end{cases}\qquad
    (H^\zeta_{ij;u,b}\zeta)_{\ell}:=  \begin{cases}
    u(\zeta_i+\zeta_j)&\ell=i,\\
    b\,T_j&\ell=j,\\
    \zeta_\ell&\text{else}.
  \end{cases}
\end{align}
\begin{remark*}[Boundary conditions in the KMP paper]\rm
  We remark that the dynamics at the boundary is slightly different from the original proposal in \cite{KMP} and is instead similar to the definition in \cite{Bertini_2007}. The generator in \cite{KMP} is given by
  \begin{align}
    \label{orig}
      L^\zeta f(\zeta)&:=\sum_{ij\in \oE}
                      \int_0^1du\,\bigl[f(H^\zeta_{ij;u}\zeta)-f(\zeta)\bigr]+\sum_{j\in \pVV} \int_0^{+\infty}db\,e^{-b}\,\bigl[f(H^\zeta_{j;b}\zeta)-f(\zeta)\bigr],
  \end{align}
where $(H^\zeta_{j;b}\zeta)_\ell=\zeta_\ell$ for $\ell\neq j$, and $(H^\zeta_{j;b}\zeta)_j=T_jb$. 
At rate 1 the edges $ij$ with $j$ in the boundary behave as the interior edges, updating both extremes to a fraction $u$ and $(1-u)$ of the sum, respectively. Furthermore, each boundary site $j$ at rate 1 updates its values to a fresh exponential random variable with mean $T_j$. Our approach works also for these boundary conditions. 
\end{remark*}

The main issue in our paper is to consider the KMP process as the marginal of a joint process in which both the oscillators and the temperatures evolve: this is explained after the next definition.\\

\begin{definition}\label{def2.1}{\bf {Joint KMP-temperature process}}

\rm We define a joint process $(X(t),T(t))_{t\ge 0}$ with $X(t)\in  \mathbb R_+^{\oVV}$ and $T(t)\in \mathbb R_+^{\oVV}$.
The variables $X:=\{X(t)\}$ evolves as a KMP process with constant (equal to 1) boundary  temperatures. The evolution of the $T:=\{T(t)\}$ variables  in the Harris graphical construction is defined by updating its values simultaneously at the Poisson events associated to the edge $ij$ to a configuration in which the temperatures in $i$ and $j$ assume an identical value suitable chosen  depending on the $X_i$ and $X_j$ variables. 
We denote by $L^{X,T}$  the generator of this process given by
	\begin{eqnarray}
	&&\hskip-2cm L^{X,T}f(X,T)=\sum_{ij\in E}\int_0^1 du \big[f\big(H_{i,j;u}(X,T)\big)- f(X,T)\big]\nn
	\\&&\hskip0.5cm +\sum_{ij\in \partial E}\int_0^\infty db\, e^{-b} \int_0^1du \big[f\big(H_{i,j;b,u}(X,T)\big)- f(X,T)\big],
	\label{gen}
	\end{eqnarray} 
where 
 \begin{align}\label{2.7a}
H_{i,j;u}(X,T)&:=
 \begin{cases}
    \big(u(X_i+X_{j}), \frac{X_i}{X_i+X_j}T_i+\frac{X_j}{X_i+X_j}T_j\big)  & \ell=i \\
   \big( (1-u)(X_i+X_{j}), \frac{X_i}{X_i+X_j}T_i+\frac{X_j}{X_i+X_j}T_j \big) & \ell=j \\   
   (X_\ell,T_\ell) &  \text{otherwise},
  \end{cases}
   \end{align}
    \begin{align}\label{2.7b}
H_{i,j;b,u}(X,T)&:=
 \begin{cases}
    \big(u(X_i+X_{j}), \frac{X_i}{X_i+X_j}T_i+\frac{X_j}{X_i+X_j}T_j\big)  & \ell=i \\
    (b,T_j)  & \ell=j\in \pVV\\
   (X_\ell,T_\ell) &  \text{otherwise}.
  \end{cases}
   \end{align}

Define the process $\zeta(t)$ by 
\begin{equation}
\zeta_i(t)=X_i(t)\,T_i(t),\qquad i\in \oVV.
\label{2.8}
\end{equation} 
\end{definition}
Note that according to the above definition $T_i(t)=T_i$ for any $t$ when $i\in \partial \VV$.
We will construct the above process in Section \ref{pron} where we prove 
our first result stated in the following Proposition.\\

\begin{proposition}
  \label{prop2.2}
 Let $(X,T,\zeta)$ be the processes of Definition \ref{def2.1}. Then, both marginals $X$ and $\zeta$ are KMP processes with generator \eqref{gen}, with boundary conditions identically equal to $1$ for $X$ and  $T_{\pVV}$ for $\zeta$.
\end{proposition}
We observe that the invariant and reversible measure for the process $X$ is  $\mu^X :=$  product of exponential random variables with mean 1.

{\bf Remark} The joint evolution of $(T,\zeta)$ can  be described as follows: the $\zeta$ marginal is the energy while the $T$ marginal is the hidden temperature. When two oscillators exchange energy the corresponding temperatures thermalise at a common value that is a deterministic function of energies and temperatures before the interaction.
Observe that the updating of the values of the temperatures at the sites $i$ and $j$  depends on the ratio $\frac{X_i}{X_i+X_j}$, if the variables $X_i$ and $X_j$ were exponential independent then this ratio has law uniform in $(0,1)$. With this in mind we introduce the following 
process.\\

\begin{definition}
\label{def2.3} \rm
We define a Markov  process $O:=(O(t))_{t\ge0}$, $O(t)\in \mathbb R_+^{\oVV}$, with boundary conditions $T_{\pVV}$ as follows. When the clock associated to $ij$ rings, $O_i$ is updated to $O'_i=U O_i + (1-U) O_j$, where $U$ is the Uniform mark associated to the Poisson ring; $O_j$ is updated to $O'_j=O'_i$ if $j$ is an interior vertex, otherwise $O_j=T_j$, that is, it keeps the boundary value. The generator of this process is
\begin{align}
L^Of(O)&:=\sum_{ij\in \oE}\int_0^1dv\, [f(H^O_{ij;v}O) - f(O)].
\label{genO}\\
 (H^O_{ij;v}O)_{\ell}&:=
  \begin{cases}
    vO_i+(1-v)O_j&\ell\in\{i,j\}\cap\VV,\\
    O_\ell&\text{else}. 
  \end{cases}\label{ho1}
\end{align}
\end{definition}

This process is in the class of the so-called opinion processes, in Section \ref{slap} we give an overview of the results on this subject. The process $O$ is a Markov process on the graph $(\oVV,\oE)$ with boundary conditions $T_{\pVV}$. Suppose individuals $i$ and $j$ have positive real opinions $O_i<O_j$ just before the clock rings for the edge $ij$. Then a random opinion $\tO$ uniformly distributed in $[O_i,O_j]$ is chosen and both $i$ and $j$ update their opinions to $\tO$, if $j$ is internal vertex; otherwise, only $i$ updates to $\tO$ while $j\in\pVV$ holds the boundary opinion $T_j$. For this reason we call $O$ an opinion model with stubborn or extremist agents, referring to the intransigent individuals at the boundary, as we have  $O_{\pVV}(t)= T_{\pVV}$ for all $t$. %

For any choice of $T_{\pVV}$, the $O$  process has a unique invariant measure that we call $\nu^{O}$ (the dependence on $T_{\pVV}$ is underlined). Indeed, the process is irreducible and the set $I^{\oVV}$ is invariant and attractive for the dynamics, where $I:=\bigl[\min_{i\in \pVV}T_i, \max_{i\in \pVV}T_i\bigr]$. We study some properties of $\nu^O$ in Section \ref{slap}.

Next Theorem, proved in Section \ref{invariant}, gives the relation between  the $(X,T)$ process and the opinion model of Definition \ref{def2.3}.\\

\begin{theorem}
  \label{T2}
   Consider the process $(X,T)$ of Definition \ref{def2.1} with $T(0)$ and $X(0)$  independent. Assume that $X(0)$ has law $\mu^X$ and that the law of $T(0)\in \mathbb R_+^{\oVV}\cap\{T_i(0)=T_i,\,\,\forall i\in \pVV\}$ is an arbitrary probability $\nu$. Then $T(t)$ and $X(t)$ are independent for each $t\ge0$, and the $T$ marginal process is distributed as the Markov process $O$ of Definition \ref{def2.3}, with initial distribution $\nu$, i.e.
  \begin{align}
  \label{mixt}
 (X(t),T(t))\text{ has law } \mu^X(d \underline x) \bigl[\nu e^{tL^{O}}\bigr](d \us).
\end{align}
\end{theorem}

Next result is about the stationary non equilibrium measure of the boundary driven KMP process.

\smallskip
\begin{definition}
\label{defM}
\rm We denote by $\cM$ the set of measures on $\R_+^{\oVV}$ which are mixtures of independent exponentials. Namely,  $\mu\in\cM$ if there exists a probability measure $\nu$ on  $\R_+^{\oVV}$ such that
\begin{align}
  \mu(d\underline \zeta) = \int \nu(d\us) \prod_{i\in\oVV} \frac1{s_i} \, e^{-\zeta_i/s_i}\,d\zeta_i.
\end{align}
 $\nu$ is called parameter measure of $\mu$. 
 \end{definition}

Our main result (Theorem \ref{teo3} below) is that the process $\zeta$ preserves $\cM$ and that the unique invariant measure for $\zeta$ belongs to $\cM$. \\

\begin{theorem}[SNS in  KMP]
  \label{teo3}\
  \nopagebreak
  
\noindent 1) The KMP process $\zeta$ preserves the set $\mathcal M$, more precisely if the law of $\zeta(0)$ belongs to $\mathcal M$ with parameter measure $\nu$ then  $\mu^\zeta_t$,
the law of $\zeta (t)$, belongs to $\cM$ with parameter measure $\nu_t=\nu e^{tL^{O}}$.

\noindent 2) The invariant measure $\mu^\zeta$ for the KMP process with boundary condition $T_{\pVV}$ belongs to $\mathcal M$, with parameter measure $\nu^{O}$, the invariant probability of the opinion model:
  \begin{equation}\label{iz1}
  \mu^\zeta(d\underline \zeta)=\int \nu^{O}(d\us)
  \,\tprod_{i\in \oVV}\,\tfrac1{s_i}\,{e^{-\zeta_i/s_i}\,d\zeta_i}.
\end{equation}
3) The correlations of $\mu_t^\zeta$ are related to those of $\nu_t$ by
\begin{align}\label{32}
  \int\mu_t^\zeta(d\underline\zeta) \tprod_{i\in \oVV}\, \zeta_i^{k_i} = \int\nu_t(d\us)\, \tprod_{i\in \oVV}\,k_i!\, s_i^{k_i},\qquad k_i\in \N_0=\{0,1,\dots\},
\end{align}
which in the limit $t\to \infty$ gives
\begin{align}
\label{corr}
  \int\mu^\zeta(d\underline\zeta) \tprod_{i\in \oVV}\, \zeta_i^{k_i} = \int\nu^O(d\us)\, \tprod_{i\in \oVV}\,k_i!\, s_i^{k_i},\qquad k_i\in \N_0=\{0,1,\dots\}.
\end{align}
\end{theorem}
Theorem \ref{teo3} will be proved in Section \ref{prteo3}.

\medskip
Thus from \eqref{corr} in order to characterize the SNS of the KMP process, we need to understand the behavior of the correlations in the stationary measure for the opinion model. We have results in this direction in  one dimension, namely   we 
 consider the one-dimensional graph
\begin{equation}
  \label{odg}
  \begin{array}{l}
      \oVV:=
    \{0,\frac1N,\dots,\frac{N-1}N,1\},\quad \pVV=\{0,1\}\\[1mm]
    E:=\{\frac{i}{N}(\frac{i+1}N):i=1,\dots,N-1\}\cup\{\frac1N 0\},\\[1mm]
T_-=T_0\quad T_+=T_1,
  \end{array}
\end{equation}
and we consider the process $O$ of Definition \ref{def2.3} in this graph. Let $\nu^O$ be the invariant measure for this process and  call 
	\begin{equation}
	\label{2.17}
C^N_{k,\ell}:= \mathbb E_{\nu^O}(O_kO_\ell)-\mathbb E_{\nu^O}(O_k)\mathbb E_{\nu^O}(O_\ell)\,.
	\end{equation}

In Section \ref{S41} we prove the following\\
\begin{proposition} \label{prop2.7}[Bounds for the one dimensional correlations]
The correlations $C^N_{k,\ell}$ of the invariant measure $\nu^O$ of the opinion model in the one-dimensional graph \eqref{odg} satisfy the bound $C^N_{k,\ell}\leq \tilde C^N_{k,\ell}$, where
\begin{align}\label{ilectilde}
\tilde C_{k,\ell}:=
\begin{cases}
\frac{(T_+-T_-)^2}{N+1}\frac kN(1-\frac \ell N)\,,& k<\ell\\[1mm]
\frac{(T_+-T_-)^2}{N+1}\frac kN(1-\frac kN)+\frac{(T_+-T_-)^2}{2N(N+1)}& k=\ell\\
\end{cases}
\end{align}
\end{proposition}
This Proposition is proved in Section \ref{pco}. The proof  uses the classical notion of duality \cite{liggett85} which relates the opinion model to a discrete KMP process,
 this is explained in Section \ref{S41}.

The bounds \eqref{ilectilde} are the key to show the one-dimensional hydrostatic limit. Let 
\begin{equation*} m(x):= T_-+x(T_+-T_-),\qquad x\in[0,1]\end{equation*} The following result is proved in Section \ref{S42}. \\
\begin{theorem}[One dimensional hydrostatic limit]
  \label{ihydr}
Consider  the process   $O(t)$ of Definition \ref{def2.3} on the one-dimensional graph \eqref{odg}, distributed with the invariant measure $\nu^O$, and denote by $\PP_{\nu^O}$ the associated probability. Then,  for any $\epsilon >0$ and for any continuous test function $\psi:[0,1]\to \mathbb R$ we have
\begin{equation}
\lim_{N\to +\infty}\PP_{\nu^O}\Bigl(\Bigl|\sum_{k=0}^N \psi(k/N)\, O^{\stat}_k-\int_{[0,1]}\psi(x)\, m(x) dx\Bigr|>\epsilon\Bigr)=0\,.
\end{equation}
\end{theorem}

We come back to general graphs and describe some additional results for the opinion model. Using the graphical construction we construct a random configuration $O$ distributed with the invariant measure $\nu^O$, as a function of a family of random walks in space-time random environment. See Proposition \ref{invo}. Then we study the process $\eta(t)$ on the space $\{0,1\}^\oE$, defined as the indicator of disagreement edges in the opinion model by
\begin{align}
  \eta_{ij}(t):= \one\{O_i(t)\neq O_j(t)\}.  
\end{align}
In particular $\eta_{ij}=1$ if $j$ is a boundary vertex. The process $\eta(t)$, called disagreement process, has interesting properties. In particular, $(\eta_{ij}(t):ij\in E')$ is Markovian for any subset $E'\subset \oE$. This allows to consider sets $E'$ with only one edge and show that the probability under $\mu^\eta$, the invariant measure for the process $\eta(t)$, satisfies $\mu^\eta(\eta_{ij}=0) =1/(n+1)$, where $n$ is the number of edges sharing a vertex with $ij$. We give also a perfect simulation algorithm of $\mu^\eta$.

Finally we show that the invariant measure for the discrete boundary driven KMP process in the same graph as the continuous one, and with the same boundary conditions, is a mixture of geometric random variables, where the parameter measure $\nu^O$ is the same as the one for the continuous case. See Theorem \ref{tdkmp} in Section \ref{sec:discrete-KMP}. To show it we couple the continuous and discrete KMP using rate one Poisson processes. 

\smallskip

The paper is organized as follows:

\noindent In Section \ref{scoupling} we illustrate the basic coupled construction of the processes $(X,T,\zeta)$ and prove Proposition \ref{prop2.2}, and Theorems \ref{T2} and \ref{teo3}.

\noindent In Section \ref{slap} we discuss in detail the opinion model $O$ and review existent related opinion models. We discuss classical and almost surely duality, prove Proposition \ref{prop2.7} and Theorem \ref{ihydr}, and study some features of the disagreement process~$\eta$.

\noindent In Section \ref{sec:discrete-KMP} we discuss the discrete KMP process.

\section{Proofs} 

In this section we construct a process 
$$(X(t),T(t),\zeta(t))=(X_i(t),T_i(t),\zeta_i(t))_{i\in \oVV }$$ which has the property that considering any pair of variables (i.e $(X,T)$, $(X,\zeta)$ or $(T,\zeta)$), the joint process on the pair is Markovian and the third collection of variables is completely determined. Our preferred viewpoint is in term of $(X,T)$, which is determined by $X$ and the initial value $T(0)$.
We construct the process $(X(t),T(t))=(X_i(t),T_i(t))_{i\in \oVV }$, $t\ge0$ in function of the initial configuration and a family of marked Poisson processes. %

Let $\cN=\cup_{ij\in \oE}\cN_{ij}$, where $(\cN_{ij})_{ij\in \oE}$, are iid Poisson processes in $\R$, with mean measure $dt$. To each $\tau\in\cN$ we associate a \emph{mark} $U(\tau)$ uniform in $[0,1]$, and for $\tau\in\cN_{ij}$ with $j\in \pVV$, associate also a mean-1 exponential random variable $B(\tau)$. Assume that the $U$ and $B$ variables are mutually independent and independent of $\cN$.\footnote{meaning that we order the time events $\cN=\{\dots,\tau_{-1}<\tau_0<0<\tau_1<\dots\}$, consider sequences of iid Uniform$[0,1]$ random variables $U_k$, and Exponential$(1)$ random variables $B_k$, all independent and independent of $\cN$ and call $U(\tau_k)=U_k$, $B(\tau_k)=B_k$.   } Denote $(\cN,\cU,\cB)$ the family of marked Poisson processes.

The initial configuration is denoted by
\begin{gather}
  \label{ic0}
  (X(0),T(0))= (X_i(0),T_i(0))_{i\in\oVV }.
\end{gather}
Let $\tau\in \cN_{ij}$, assume $X(\tau-),T(\tau-)$ are defined, and denote
\begin{align}
  \label{tau4}
&U:=U(\tau),\quad B:=B(\tau),\quad  X := X(\tau-),\quad  T := T(\tau-),\quad V:= \frac{X_i}{X_i+X_j},
\end{align}
where we can ignore $B$ if $j\notin\pVV$.

For each $\ell\in\oVV $, update $X$ and $T$ by
\begin{align}\label{xit2}
  X_\ell(\tau)&:=
  \begin{cases}
    U(X_i+X_{j})  & \ell=i \\
    (1-U)(X_i+X_{j})  & \ell=j\text{ and } j\in \VV\\
    B  & \ell=j\text{ and } j\in \pVV\\
    X_\ell &  \text{otherwise},
  \end{cases}\\
  T_\ell(\tau)&:=
   \begin{cases}
     \displaystyle{ V T_i+ (1-V) T_j}& \ell\in\{i,j\} \cap\VV\\
        T_\ell  & \text{otherwise}.
  \end{cases}\label{tell1}
   \end{align}
Then, for $t\ge 0$,  define
\begin{align}
  \label{xttt}
  (X(t),T(t))&:=\bigl(X(\tau^t),T(\tau^t)\bigr), \\
  \tau^t&:=\sup\bigl( \cN\cap[0,t]\bigr), \label{tau1}
\end{align}
with the convention $\sup(\emptyset)=0$. The construction is performed recursively. Set $X(t)=X(0)$ for $t<\tau_1:=\min(\cN\cap\R^+)$, use the recipes \eqref{xit2}-\eqref{tell1} to evolve from $\tau_1-$ to $\tau_1$, use \eqref{xttt} for $t\in[\tau_1,\tau_2)$, where  $\tau_2:=\min(\cN\cap(\tau_1,\infty))$, and so on.
Finally, for $t\ge0$, define the $\zeta$ process by
\begin{align}
  \label{zetat1}
  \zeta_i(t) := X_i(t)\,T_i(t),\quad i\in\oVV .
\end{align}
When $j\in\pVV$, $\zeta_j(t)$ is a mean $T_j$ exponential random variable, it is updated to a fresh variable with the same law at each $\tau\in\cN_{ij}$, and it is never affected by the values of the other coordinates.

We have defined the process $(X,T,\zeta)$ as a function of $(X(0),T(0),\zeta(0))$ and $(\cN,\cU,\cB)$. We say that $(\cN,\cU,\cB)$ \emph{governs} the process $(X,T,\zeta)$.

The $(X,\zeta)$ marginal of the process $(X,T,\zeta)$ is a coupling between $X$, an equilibrium KMP with constant unitary boundary conditions, and $\zeta$, a KMP with general boundary conditions. The coupling is obtained by letting both marginals be governed by the same marked Poisson processes $(\cN,\cU,\cB)$. 

\subsection{Proof of Proposition \ref{prop2.2}}
\label{pron}

We use the previous construction of the process and recalling \eqref{tau4} we observe that, since $B$ is exponential with mean $1$, the $X$ marginal is Markov with generator \eqref{Lzeta} and $T_{\pVV}\equiv 1$. To treat the $\zeta$ marginal, let $\tau\in\cN_{ij}$, recall notation \eqref{tau4} and denote $\zeta_k=\zeta_k(\tau-)=X_kT_k$ for $k\in \oVV $. By \eqref{xit2} and \eqref{tell1}, each vertex $\ell\in\oVV $ is updated  to
\begin{align}
  \zeta_\ell(\tau)
  &=
    \begin{cases}
  \bigl( VT_i+(1-V)T_{j}\bigr) U (X_i+X_{j}) & \ell=i\\
 \bigl( VT_i+(1-V)T_{j}\bigr) (1-U) (X_i+X_{j}) & \ell=j\text{ and } j\in \VV\\
      BT_\ell& \ell=j\text{ and } j\in \pVV\\
 X_\ell T_\ell&\text{otherwise}
\end{cases}\label{ze1}\\[2mm]
&=
\begin{cases}
 U \bigl(X_i T_i+ X_{j}T_{j}\bigr) & \ell=i\\
(1-  U) \bigl(X_i T_i+ X_{j}T_{j}\bigr) & \ell=j\text{ and } j\in \VV\\
      BT_\ell& \ell=j\text{ and } j\in \pVV\\
  X_\ell T_\ell&\text{otherwise}.%
\end{cases}\label{ze2}
\end{align}
To conclude, observe that $X_iT_i=\zeta_i(\tau-)$. \qed

\subsection{Proof of Theorem \ref{T2}}
\label{invariant}

We provide two proofs. One based on the graphical construction and the other is an analytical computation.
\begin{proof}[Graphical proof of Theorem \ref{T2}]
Since  $\mu^X$ is reversible for the $X$ marginal process, we have that $X(t)$ has law $\mu^X$ for all $t$.

Let $\tau\in\cN_{ij}$ and recall the notation \eqref{tau4}. We claim that the random elements
\begin{align}
  \label{indep9}
  U,\quad B,\quad V,\quad   X_i+X_{j},\quad  (X_{k})_{k\in\oVV \setminus\{i,j\}}\quad
  \text{ are mutually independent}.
  \end{align}
Since $U$ and $B$ are independent fresh marks attached to $\tau$, we have that $(U,B)$ is independent of the remaining elements, as they are functions of $(X(t))_{t<\tau}$. Since $X_i$ and $X_j$ are iid exponential, $V$ has uniform distribution in $[0,1]$ and $V$ is independent of $X_i+X_{j}$. This and the fact that $X(\tau-)$ are iid exponential show \eqref{indep9}.

  In turn, \eqref{indep9} and \eqref{xit2} imply that $V(\tau)$ and $X(\tau)$ are independent, which by (our Markovian) construction  implies that $V(\tau)$ is independent of $(X(t))_{t\ge\tau}$. Let  $\tau,\ttau\in\cN$ such that $\ttau>t\ge \tau$; we have that $V(\ttau)$ is a function of $X(t)$ and $\cN\cap (t,\infty]$. Hence, $\{V(\tau):\tau\in\cN,\,\tau> t\}$ and  $\{V(\tau):\tau\in\cN,\,\tau\le t\}$ are independent, implying $\cV:=(V(\tau))_{\tau\in\cN}$ are iid uniform in $[0,1]$, and independent of $\cN$.

By definition \eqref{tell1}, the $T$ process is governed by the marked Poisson processes $(\cN,\cV)$ so that it is a Markov process and its law coincides with the opinion model $O$ with the same initial condition.

The independence of $X(t)$ and $T(t)$ follows by induction. Assume that $X(\tau-)$ and $T(\tau-)$ are independent. Use the definitions \eqref{xit2} and \eqref{tell1} and the independence \eqref{indep9} to see that also  $X(\tau)$ and $T(\tau)$ are independent. To conclude, use the definition \eqref{xttt}.
\end{proof}
\begin{proof}[Analytical proof of Theorem \ref{T2}]
We show that the time dependent probability measure $\mu^X(d \underline{x}) \nu_t(d\underline{s})=\mu^X(d \underline{x}) \nu e^{tL^O}(d\underline{s})$ with $\nu$ an arbitrary initial condition, solves the Kolmogorov equation for the coupled $(X,T)$ dynamics. Recalling the Definition \ref{gen} of the  coupled generator $L^{X,T}$,  
this is equivalent to show that for any function $f=f(\underline x, \underline s)$, 
$\underline x\in  \mathbb R_+^{\oVV}$,  $\underline s\in \mathbb R_+^{\oVV}$, we have
\begin{equation}\label{kolmo}
\frac{d}{d t}\left[\int \mu^X(d \underline{x}) \nu_t(d\underline{s}) f(\underline x, \underline s)\right]=
\int \mu^X(d \underline{x}) \nu_t(d\underline{s}) L^{X,T} f(\underline x, \underline s)\,.
\end{equation}
If we call $\tilde f =\tilde f(\underline s):=\int \mu^X(d \underline{x})f(\underline x, \underline s)$, the validity of the joined Kolmogorov equation \eqref{kolmo} follows by the identity
\begin{equation}\label{formulachiave}
\int \mu^X(d \underline{x})L^{X,T} f(\underline x, \underline s)=L^O \tilde f(\underline s)\,.
\end{equation}
Once  \eqref{formulachiave} is proved,  \eqref{kolmo} becomes
\begin{equation}\label{kolmoO}
\frac{d}{d t}\left[\int \nu_t(d\underline{s}) \tilde f(\underline s)\right]=
\int \nu_t(d\underline{s}) L^{O} \tilde f(\underline s)\,,
\end{equation}
which is satisfied since it is exactly the Kolmogorov equation for the opinion model.

\smallskip
We prove then \eqref{formulachiave}. For simplicity of notation we show the basic computations distinguishing the case of an internal edge and the case of a boundary edge and considering in both cases functions that depend just on the values of the variables on the two nodes belonging to the edge.

The first basic computation, related to the dynamics on internal edges, is the following.
Consider $ij\in E$, $f=f(x_i,x_j,s_i,s_j)$ and the following formula
\begin{align}
  &\int_{\mathbb R_+^2}e^{-(x_i+x_j)}dx_idx_j\, \int_0^1 du \notag \\
  &\times f\Big(u(x_i+x_j),\, (1-u)(x_i+x_j),\, \frac{s_ix_i}{x_i+x_j}+\frac{s_jx_j}{x_i+x_j},\, \frac{s_ix_i}{x_i+x_j}+\frac{s_jx_j}{x_i+x_j}\Big)\,.\label{contointerno}
\end{align}
We perform the change of variables whose Jacobian is unitary
\begin{equation}
\left\{
\begin{array}{l}
y_i=u(x_i+x_j)\,,\\
y_j=(1-u)(x_i+x_j)\,,\\
v=\frac{x_i}{x_i+x_j}\,,
\end{array}
\right.
\end{equation}
and \eqref{contointerno} becomes
\begin{align}
& \int_{\mathbb R_+^2}e^{-(y_i+y_j)}dy_idy_j \int_0^1 dv f\Big(y_i,\, y_j,\, s_iv+s_j(1-v),\, s_iv+s_j(1-v)\Big)\,\\
& =\int_0^1 dv \,\tilde f\Big(s_iv+s_j(1-v),\, s_iv+s_j(1-v)\Big)\,.
\end{align}

The second basic computation, related to the dynamics on boundary edges, is the following.
Consider $ij\in \partial E$, $f=f(x_i, x_j, s_i)$ and compute
\begin{equation}\label{contoesterno}
\int_{\mathbb R_+^2}e^{-(x_i+x_j)}dx_idx_j \int_0^{+\infty} e^{-b} db\int_0^1 du f\Big(u(x_i+x_j), b, \frac{s_ix_i}{x_i+x_j}+\frac{T_jx_j}{x_i+x_j}\Big)\,.
\end{equation}
We perform the change of variables whose Jacobian is unitary
\begin{equation}
\left\{
\begin{array}{l}
y_i=u(x_i+x_j)\,,\\
y_j=b\,,\\
B=(1-u)(x_i+x_j)\,,\\
v=\frac{x_i}{x_i+x_j}\,,
\end{array}
\right.
\end{equation}
and \eqref{contointerno} becomes
\begin{align}
& \int_{\mathbb R_+^2}e^{-(y_i+y_j)}dy_idy_j \int_0^{+\infty}e^{-B}dB\int_0^1 dv f\Big(y_i, y_j, s_iv+T_j(1-v) \Big)\,\\
& =\int_0^1 dv \tilde f\Big(s_iv+T_j(1-v)\Big)\,.
\end{align}
Using these two basic computations it is easy to deduce the validity of \eqref{formulachiave}.
\end{proof}

\subsection {Proof of Theorem \ref{teo3}}
\label{prteo3}

1) Assume the law $\mu$ of $\zeta(0)$ is in $ \mathcal M$ with parameter measure $\nu$. Then $\zeta(0)$ is equal in law to  the product $X(0)T(0)$, 
where $(X(0),T(0))$  has law $\mu^X(d\underline x)\nu(d \underline s)$. In fact $\zeta(0)=X(0)T(0)$ has distribution $\mu$ given on test functions $\varphi:\R^{\oVV}\to\R$ by
\begin{align}
\nonumber 
\int\mu(d \underline \zeta) \varphi(\zeta)
&= \int \nu(d\us)
\Bigl(\tprod_{i\in\oVV} \int  e^{-x_i} dx_i \Bigr) \,\varphi(\ux\circ\us)\\
&= \int \nu(d\us)
\Bigl(\tprod_{i\in\oVV} \int \frac1{s_i} e^{-\zeta_i/s_i} d \zeta_i \Bigr) \,\varphi(\underline \zeta)\,,
\label{mu2}
\end{align}
where $\ux:=(x_i)_{i\in\oVV}$ and $\ux\circ\us:=(x_is_i)_{i\in\oVV}$ is the Hadamard product of vectors. This means that $\mu\in \mathcal M$ with parameter measure $\nu$. The KMP process with initial distribution $\mu\in \mathcal M$ can be realized as the third component of the triple $(X^{\stat},T,\zeta)$ and for each $t$, $X^{\stat}(t)$ and $T(t)$ are independent. By Theorem \ref{T2}, 
the law of $T(t)$ is $\nu e^{tL^{O}}$. We conclude using the computation in \eqref{mu2}.

Item 2) follows directly from item 1).  Item  3) is immediate from \eqref{mixt} recalling that  $\mathbb E(Y^k)=k!\lambda^k$ if  $Y$ is an exponential random variable of mean $\lambda$. \qed

\section{Opinion models}
\label{slap}
In this section we study the opinion model $O(t)$  in the graph $(\oVV,\oE)$ with generator \eqref{genO} and fixed boundary conditions $T_{\pVV}$. We give a short background on opinion models. In Section \ref{S41} we give a bound for the correlations in the one dimensional lattice, and in  Section \ref{S42} we use it to deduce a linear hydrostatic behavior for the empirical measure associated to the invariant measure.

The remaining of this section describes additional results on the invariant measure for the process $O(t)$. In Section \ref{cwalk3} we study the evolution of the disagreement process, a process that indicates the edges whose vertices have different opinions. This process is Markov and behaves as an spiking neuron process, see Section \ref{S44}. We compute explicitly the density of pairs with disagreement opinions under the invariant measure $\nu^O$ and perform a perfect simulation of the invariant measure for the disagreement process.

\paragraph{Background} The process $O(t)$ is a variation of a big class of opinion processes, see the survey by Castellano, Fortunato and Loreto \cite{castellano2009}. Positive numbers representing opinions sit at the vertices of a graph. At interaction times of an edge, the involved individuals update their opinions to lower their opinion difference. In the Deffuant model, introduced in \cite{deffuant2000mixing}, the updating occurs only when the modulus of their current opinion difference does not superate a threshold parameter.
Extremists are individuals that do not update the new agreed value after the interaction. Weisbuch, Deffuant and Amblart \cite{weisbuch2005persuasion} include extremists, and show that the population splits in two, each piece taking the opinion of one of the extremists. Vazquez, Krapivsky and Redner \cite{vazquez-krapivsky-redner,zbMATH02010952} consider a population with three opinions left/center/right, where center interacts with left and right, but they do not interact between them. The final state is either dominated by centrists, or split population in leftists and rightist or all population goes to one of the extremes. Lanchier, \cite{zbMATH06235591}
and Lanchier and Li \cite{lanchier-li-2020} study consensus in Deffuant models in infinite graphs. Gantert, Heydenreich, Hirscher \cite{ghh2020} study local and global agreement in the Deffuant model. Hirscher \cite{Hirscher-2017} study Deffuant in infinite volume with extremists (called overlay determined individuals). Haggstrom \cite{hagsstrom-2012} studies phase transitions in the Deffuant model in $\Z$. G{\'o}mez-Serrano, Graham and Le Boudec \cite{zbMATH06092639} prove convergence of the empirical measure of the Deffuant model to a limit satisfying a differential equation, whose solutions concentrate around several finite opinions, and at last converge to consensus.  The opinion is a conserved quantity in the models of the above papers. Dag{\`e}s and Bruckstein \cite{db-2022} consider a complete graph, and at each interaction time update the involved vertices with two independent random variables, uniformly distributed between the current opinions; this is a non conservative dynamics.

In the gossip model, also called averaging process, the interacting vertices take the mean value of their opinions. Picci and Taylor \cite{picci-taylor2013} show convergence of the gossip model to consensus. The papers by Boyd, Ghosh, Prabhakar, Shah \cite{1498447,1638541} study the speed of convergence to the average, and the survey Shah \cite{zbMATH05659657} study computational aspects. The survey of Aldous and Lanoue \cite{zbMATH06050911} study the convergence towards the invariant configuration. Aldous \cite{zbMATH06216071} studies the asymptotics, correlations and duality for several models including averaging processes in general graphs. Caputo, Quattropani and Sau \cite{zbMATH07733575} study cutoff phenomena; Quattropani, Sau \cite{zbMATH07692286} study mixing and duality; Sau \cite{sau2024tiny} study fluctuations.

Our model is similar to the gossip/averaging process, but when two individuals interact, they choose a random value uniformly distributed between their current opinions, and both adopt that value. This is a non-conservative dynamics. We also have extremists at the boundary vertices, with different opinions. This implies that the invariant measure is non trivial, a fact also present in the gossip/averaging process, but not studied there, as far as we know.

\paragraph{Invariant measure in the single individual case}
Consider a single individual in contact with two external opinions $T_\pm$. We have $\VV=\{1\}$, $\partial \VV=\{0,2\}$ and a unique random variable $O_1$. We show that the unique invariant measure has support on $[T_-,T_+]$ and its density is given by
\begin{equation}\label{arcsin}
\rho(y)=\frac{1}{\pi\sqrt{(y-T_-)(T_+-y)}}\,, \qquad y\in[T_-,T_+],
\end{equation}
a linear transformation of a Beta$(\frac12,\frac12)$ distribution or Arcsin law.
The only oscillator performs a Markov chain that from a state $y$ jumps at rate $1$ to a point uniformly distributed in the interval $[T_-,y]$, and at the same rate to a point uniformly distributed in $[y,T_+]$. The invariant measure $\rho(y)dy$ of such process is characterized by the balance equation
\begin{equation}\label{balance}
2\rho(y)=\int_{T_-}^{y} \frac{d x}{T_+-x}+\int^{T_+}_{y} \frac{d x}{x-T_-}\,.
\end{equation}
Differentiating \eqref{balance} with respect to $y$ we get
$$
2\rho'(y)=\frac{1}{T_+-y}-\frac{1}{y-T_-}\,,
$$
which can be solved by separation of variables and has as unique positive normalized to one solution given in \eqref{arcsin}.

\subsection{Classical duality and correlations}
\label{S41}
In this subsection we introduce a classic duality relation, as in Liggett \cite{liggett85},  between the opinion process and the discrete KMP process introduced in \cite{KMP}. This will be useful to compute the correlations of the invariant measure for the opinion process, which in turn are used to show the hydrostatic limit.

We consider two versions of the discrete KMP process. The version with particles absorbed at the boundary will be dual of the opinion model. The boundary driven version will be coupled with the corresponding continuous one in Section \ref{sec:discrete-KMP}.

The discrete KMP process $K(t)$ with boundary conditions $T_{\pVV}$ is analogous to the original real valued KMP process with generator \eqref{Lzeta}, but with state space $(\mathbb N_0)^{\oVV}$ . When the Poisson event $\tau$ belongs to $\cN_{ij}$, for an internal $ij$, the value $K_i$ is updated to a new value $K'_i$ chosen uniformly in  $\{0,\dots,K_i+K_j\}$ and the value $K_j$ is updated to $K'_j=K_i+K_j-K'_i$. When $j$ is in the boundary, the value of $K_i$ is updated in the same way, but the value at the boundary $K_j$ is updated to a fresh geometric random variable in $\mathbb N_0=\{0,1,\dots\}$ with mean $T_j$. The generator of the discrete KMP with boundary conditions $T_{\pVV}$ is given by
\begin{align}
L^{K}f(\uk)&:=\sum_{ij\in \oE} \frac{1}{k_i+k_j+1}
\sum_{h=0}^{k_i+k_j} \Bigl( \one\{j\in\VV\}\, [f(H_{ij;h}(\uk))-f(\uk)]
\label{genK}  \\
&\qquad+  \one\{j\in\pVV\}\,
\sum_{b\ge0}  \bigl(\tfrac{T_j}{T_j+1}\bigr)^b
\tfrac{1}{T_j+1} [f(H_{ij;h,b}(\uk))-f(\uk)]\Bigr)
\,,\nonumber \\
\label{genK1}
(H_{ij;h}(\uk))_\ell&:=  \begin{cases}
h &\ell=i\\
k_i+k_j-h & \ell=j \\
k_\ell&\text{else},%
\end{cases}\\
(H_{ij;h,b}(\uk))_\ell&:= \begin{cases}
b&\ell=j\text{ and } j\in \pVV\\
(H_{ij;h}(\uk))_\ell&\text{else}.%
\end{cases}\label{ze3}
\end{align}

The discrete KMP process with absorbing boundary conditions has the same dynamics in the internal edges but particles are absorbed at the boundary vertices $\pVV$. The generator of the absorbed dynamics, that we denote by $L^K_0$, is
\begin{align}
L^{K}_0f(\uk)&:=\sum_{ij\in E} \frac{1}{k_i+k_j+1}
\sum_{h=0}^{k_i+k_j}  [f(H_{ij;h}(\uk))-f(\uk)]
\label{genK1-absorbed} \\
&\qquad+  \sum_{ij\in \partial E} \frac{1}{k_i+1}
\sum_{h=0}^{k_i}  [f(H_{ij;h}(\uk))-f(\uk)]\,.
\label{bo}
\end{align}

The absorbed discrete KMP defined in \eqref{genK1-absorbed} is slightly different from the one defined in \cite{KMP}, where for each boundary edge $ij\in \partial E$ at rate one all the particles $K_i$ that are at $i\in \VV$ jump to $j\in \partial \VV$ and stay there forever.

\paragraph{Classical Duality}
The next result shows that opinion model and the discrete KMP model absorbed at the boundary are dual of each other in the sense of Definition 3.1 in Liggett \cite{liggett85}.
\\
\begin{lemma} For any $O=(O_i)_{i\in \oVV}\in (\mathbb R^+)^{\oVV}$, such that $O_i=T_i$ when $i\in \partial \VV$, and for any $K=(K_i)_{i\in \oVV}\in (\mathbb \mathbb N_0)^{\oVV}$ we have for any $t\geq 0$
\begin{equation}\label{dualt}
\mathbb E_O\Bigl(\prod_{i\in \oVV}O_i(t)^{K_i}\Bigr)=\mathbb E_K\Bigl(\prod_{i\in \oVV}O_i^{K_i(t)}\Bigr)\,,
\end{equation}
where on the left hand side we have the expected value with respect to the opinion model $O(t)$ with initial condition $O$ and $K$ is fixed, while on the right hand side we have the expected value with respect to the absorbed discrete KMP process $K(t)$, defined via \eqref{genK1-absorbed}, with initial condition $K$ and $O$ is fixed.

\end{lemma}

\begin{proof}
 To show the validity of \eqref{dualt}, as illustrated in \cite{liggett85}, we need just to show that for any $O,K$ 
\begin{equation}\label{dualstat}
L^O\Bigl(\prod_{i\in \oVV}O_i^{K_i}\Bigr)=L^K_0\Bigl(\prod_{i\in \oVV}O_i^{K_i}\Bigr)\,,
\end{equation}
where the generator $L^O$ acts on the variables $O$ and the generator $L^K_0$ acts on the variables $K$. Relation \eqref{dualstat} follows directly by the following computations that are easily derived using
\begin{align}
\int_0^1dv \, v^k(1-v)^m=\frac{1}{(m+k+1)\binom{m+k}{k}}\,.
\end{align}
For $ij\in E$ we have
\begin{align}
&\int_0^1 dv\, \left(vO_i+(1-v)O_j\right)^{K_i+K_j}\nonumber \\
&=\sum_{m=0}^{K_i+K_j}\binom{K_i+K_j}{m}\int_0^1 dv\,\big(vO_i\big)^m\big((1-v)O_j\big)^{K_i+K_j-m}\nonumber \\
&=\sum_{m=0}^{K_i+K_j}\frac{1}{K_i+K_j+1}O_i^mO_j^{K_i+K_j-m}\,.
\end{align}
For $ij\in \partial E$ we have
\begin{align}
&\int_0^1 dv\, \left(vO_i+(1-v)T_j\right)^{K_i}T_j^{K_j}\nonumber \\
&=\sum_{m=0}^{K_i}\binom{K_i}{m}\int_0^1 dv\,\big(vO_i\big)^m\big((1-v)T_j\big)^{K_i-m}T_j^{K_j}\nonumber \\
&=\sum_{m=0}^{K_i}\frac{1}{K_i+1}O_i^mT_j^{K_i+K_j-m}\,.
\end{align}
Formulas \eqref{dualstat} can be directly derived from the above computations.
\end{proof}

\begin{remark} \rm The duality relation between continuous and discrete KMP processes proved  in \cite{KMP}  is
\begin{equation}\label{dualkmp}
\mathbb E_\zeta\Bigl(\prod_{i\in \oVV}\zeta _i(t)^{K_i}/K_i!\Bigr)=\mathbb E_K\Bigl(\prod_{i\in \oVV}\zeta_i^{K_i(t)}/K_i(t)!\Bigr)\,,
\end{equation}
where on the left hand side the expectation is with respect to the continuous KMP process $\zeta(t)$ with initial condition $\zeta$ and $K$ fixed, while on the right hand side we have the expected value with respect to the absorbed discrete KMP process $K(t)$, defined via \eqref{genK1-absorbed}, with initial condition $K$ and $\zeta$ is fixed.

If in \eqref{dualkmp} we consider the initial condition $\zeta$ distributed according to $\mu$ which is a product of exponential distributions of mean $O$ we get
\begin{equation}\label{dualkmpp}
\int \mu(d\zeta)\mathbb E_\zeta\Bigl(\prod_{i\in \oVV}\zeta _i(t)^{K_i}/K_i!\Bigr)=\mathbb E_K\Bigl(\prod_{i\in \oVV}O_i^{K_i(t)}\Bigr)\,,
\end{equation}
Using \eqref{32} on the left hand side we get the duality relation \eqref{dualt}.
\end{remark}
\paragraph{Two point correlations in the one dimensional lattice}
Using the duality relation \eqref{dualt} and  considering the limit $t\to +\infty$, we can compute the stationary moments of the opinion model from the asymptotic behavior of the discrete KMP model absorbed at the boundary. We consider here the case of the one dimensional lattice \eqref{odg} and moments of order $n=1,2$.

Consider the lattice $\VV=\{1,\dots,N-1\}$ and $\pVV=\left\{0,N\right\}$. Let us call 		\begin{equation}
\label{4.15}m^N_k:=\EE O^{\stat}_k,\quad M^N_{k,\ell}:=\EE\bigl( O^{\stat}_k\,O^{\stat}_\ell\bigr),\quad C^N_{k,\ell}:=M^N_{k,\ell}-m^N_km^N_\ell.
\end{equation}
 Since the covariances $C^N_{k,\ell}$ are symmetric in $k,\ell$ we can restrict to the region $0\leq k\leq \ell\leq N$.

Using the duality of the previous paragraph, we obtain that $m^N$ is related to the asymptotic  behavior of the discrete KMP process with one single particle. In this case we can identify $K(t)\equiv\left\{y_1(t)\right\}$, where $y_1(t)$ is the position at time $t$ of the only particle present; this is a simple absorbed random walk and
$m_k^N=\mathbb E_k(T_{y_1(\tau)})$ where $\tau=\inf\left\{t; y_1(t)\in \pVV\right\}$, and in particular solves the discrete harmonic problem
\begin{equation}\label{linear}
\left\{
\begin{array}{l}
m^N_k=\frac 12 m^N_{k-1}+\frac 12 m^N_{k+1}\,,\\
m^N_0=T_{-}\,, m^N_N=T_+\,,
\end{array}
\right.
\end{equation}
whose solution is $m^N_k=T_-+\frac{\left(T_+-T_-\right)}{N}k$.

The function $M^N$ is related to the asymptotic behavior of the boundary absorbed discrete KMP with two particles. In this case we can identify $K(t)\equiv\left\{y_1(t),y_2(t)\right\}$, where $y_1(t)\leq y_2(t)$ are the positions at time $t$ of the two particles; this can be easily shown to be a two dimensional inhomogeneous random walk
on the triangular region $\left\{(k,\ell)\in \mathbb Z\times \mathbb Z\,: 0\leq k\leq\ell\leq N\right\}$ which is absorbed at the 3 vertices. By the duality we have
\begin{equation}\label{4.17}
M_{k,\ell}^N=\mathbb E_{k,\ell}\big(T_{y_1(\tau)}T_{y_2(\tau)}\big),
\end{equation} 
where $\tau=\inf\left\{t\,:\, y_i(t)\in \pVV\,, i=1,2\right\}$. 

In particular we have that $M^N$ solve a bi-dimensional harmonic problem. The boundary conditions are fixed as
$M^N_{0,\ell}=T_-m^N_\ell$ and $M^N_{k,N}=m^N_kT_+$ since for these boundary points one of the two dual particles does not evolve and the values can be obtained solving a one dimensional problem like \eqref{linear}. Detailing the transition rates, a simple computation gives that the associate discrete problem is the following
\begin{equation}\label{disc2}
\begin{cases}
\frac 14\bigl(M^N_{k-1,\ell}+M^N_{k+1,\ell}+M^N_{k,\ell-1}+M^N_{k,\ell+1}\bigr)=m^N_{k,\ell} & 1<k+1< \ell<N\,,\\
\frac {3}{10}\bigl(M^N_{\ell-2,\ell}+M^N_{\ell-1,\ell+1}\bigr)+\frac 15\bigl(M^N_{\ell-1,\ell-1}+M^N_{\ell,\ell}\bigr)= M^N_{\ell-1,\ell} & 1<\ell<N\,,\\
\frac 14\bigl(M^N_{k-1,k-1}+M^N_{k-1,k}+M^N_{k+1,k+1}+M^N_{k,k+1}\bigr)= M^N_{k,k} & 0<k<N\,,\\
M^N_{0,\ell}=T_-m^N_\ell & 0\leq \ell \leq N\,,\\
M^N_{k,N}=m^N_kT_+ & 0\leq k\leq N\,.
\end{cases}
\end{equation}
The above equations are equivalent to those for the correlations of the KMP model in section 2.4 of \cite{Bertini_2007}.

\subsection{ Proof of Proposition \ref{prop2.7}}
\label{pco}
As in \cite{Bertini_2007}, it is not possible to write in a simple closed form the solution of \eqref{disc2} but it is possible to get a simple solution for each $N$ with a slight modification of the boundary conditions. The boundary conditions are the values fixed at the two catheti of the triangular region and are fixed by the values on the last two lines in \eqref{disc2}. We call $\tilde M^N_{k,\ell}$ the solution of this modified problem. We change the values of the boundary conditions  on just 2 points of the boundary that are 2 of the 3 vertices of the triangle. In particular we let the values $M_{0,0}^N$ and $M^N_{N,N}$ to be free parameters differently from \eqref{disc2} where they are fixed as $M_{0,0}^N=T_-^2$ and $M^N_{N,N}=T_+^2$.
Using these two free parameters we can search for a solution in the form
\begin{equation}\label{lasol}
\tilde M^N_{k,\ell}=\left\{
\begin{array}{ll}
Ak+B\ell +C k\ell +D\,,& k<\ell\\
Ck^2+(A+B)k+E\,, & k=\ell\,,
\end{array}
\right.
\end{equation}
and we find a solution provided
\begin{gather}\label{ipar}
A=\frac{(T_+-T_-)(T_++NT_-)}{N(N+1)}\,,\quad
B=T_-\frac{(T_+-T_-)}{N}\,,\quad
C=\frac{(T_+-T_-)^2}{N(N+1)}\,,\\
D=T_-^2\,,\quad
E=\frac{(T_+-T_-)^2}{2N(N+1)}+T_-^2\,.\label{ipar1}
\end{gather}
It is indeed not difficult to check that \eqref{lasol} with the coefficients fixed as in \eqref{ipar}-\eqref{ipar1} solves
\eqref{disc2} apart the fact that we have $\tilde M_{0,0}^N=\frac{(T_+-T_-)^2}{2N(N+1)}+T_-^2>M_{0,0}$ and $\tilde M^N_{N,N}=\frac{(T_+-T_-)^2}{2N(N+1)}+T_+^2>M_{N,N}$.

Let us introduce the function $F_N:\pVV\times\pVV\to \mathbb R^+$ defined by
\begin{equation}
\label{4.23}
F_N(k,\ell)=
\begin{cases}
\frac{(T_+-T_-)^2}{2N(N+1)}+T_-^2 & \textrm{if}\ k=\ell=0\\[1mm]
\frac{(T_+-T_-)^2}{2N(N+1)}+T_+^2 & \textrm{if}\ k=\ell=N\\[1mm]
T_-T_+ & \textrm{if}\ k,\ell\in \pVV\,, k\neq \ell\,.
\end{cases}
\end{equation}
Notice that $\tilde M^N$ has also a probabilistic representation 
\begin{equation}\label{4.24}
\tilde M_{k,\ell}^N=\mathbb E_{k,\ell}\Big(F_N(y_1(\tau),y_2(\tau))\Big).
\end{equation}
Since $F_N(k,\ell)\geq T_kT_\ell$ where 
\begin{equation}
\label{4.23b}T_kT_\ell=
\begin{cases}T_-^2 & \textrm{if}\ k=\ell=0\\[1mm]
T_+^2 & \textrm{if}\ k=\ell=N\\[1mm]
T_-T_+ & \textrm{if}\ k,\ell\in \pVV\,, k\neq \ell\,.
\end{cases}
\end{equation}  Using the two probabilistic representations \eqref{4.17} and \eqref{4.24} we get  that
$\tilde M_{k,\ell}^N\geq M_{k,\ell}^N$. The result now follow by the fact that \eqref{ilectilde} coincide with
$\tilde C_{k,\ell}=\tilde M_{k,\ell}^N-m^N_km^N_\ell$.
\qed 
\\
\begin{remark}\rm
The numbers $\tilde M^N_{k,\ell}$ and $\tilde C_{k,\ell}$ are respectively the moments and the covariances of the invariant measure of a modified opinion model $\tilde O$. The dynamics of $\tilde O$ in the internal edges is the same as that of $O$. At the boundary, when an exponential clock $\tau \in \mathcal N_{i,j}$ rings ($i,j$ is therefore either $(N-1),N$ or $1,0$), the variable $\tilde O_j$ is updated to $\tilde O'_j$ a fresh independent random variable (sampled from an i.i.d. sequence independent from everything) having mean value $T_j$ and variance $\frac{(T_+-T_-)^2}{2N(N+1)}$; while instead the random variable $\tilde O_i$ is updated to $U\tilde O_i+(1-V)\tilde O'_j$ where $U$ is the uniform mark associated to $\tau$. This can be proved writing the conditions $\mathbb E\bigl(L^{\tilde O}[\tilde O_k^{\textrm{stat}}\tilde O_\ell^{\textrm{stat}}]\bigr)=0$
for $0\leq k\leq \ell\leq N$, that coincide with the modified version of the system \eqref{disc2}.
\end{remark}

\subsection {Proof of Theorem \ref{ihydr}} 
\label{S42}
Using Proposition \ref{prop2.7} we can prove the hydrostatic behavior of the opinion model in the one dimensional graph \eqref{odg}. 

Consider the empirical measure $\pi_N(O)$ that is a positive measure on the real interval $[0,1]$ defined by
\begin{equation}
\pi_N(O)=\frac 1N \sum_{k=0}^N O_k \delta_{k/N}\,,
\end{equation}
where $\delta_a$ is the delta measure at $a\in[0,1]$ so that given a continuous function $\psi:[0,1]\to \mathbb R$ we have
$$
\int_{[0,1]}\psi \, d\pi_N(O)=\frac 1N\sum_{k=0}^N\psi(k/N)O_k\,.
$$
To prove Theorem \ref{ihydr} we now prove  that this sequence of random measures converges weakly in probability to an absolutely continuous measure having linear density $m(x):=T_-+x(T_+-T_-)$.

\medskip

We first observe that 
\begin{equation*}
\lim_{N\to\infty}  \mathbb E_{\nu^O}\Bigl(\int_{[0,1]}\psi\, d\pi_N(O)\Bigr)=
\lim_{N\to\infty} \frac 1N\sum_{k=0}^Nm^N_k\psi(k/N) =\int_{[0,1]}\psi(x)\, m(x) dx.
\end{equation*}
 Hence, we need to prove
\begin{equation}\label{farce}
\lim_{N\to +\infty}\mathbb P_{\nu^O}\Bigl(\Bigl|\int_{[0,1]}\psi \, d\pi_N(O)-\mathbb E_{\nu^O}\Bigl(\int_{[0,1]}\psi \, d\pi_N(O)\Bigr)\Bigr|>\epsilon\Bigr)=0\,.
\end{equation}
To prove \eqref{farce} we apply Chebysev inequality. By linearity, splitting a function in its positive and negative part, we can without loss of generality restrict to the case of positive test functions $\psi\geq 0$. Applying Chebysev inequality we have that 
\begin{equation}
\label{4.28a}
\mathbb P_{\nu^O}\Bigl(\Bigl|\int_{[0,1]}\psi \, d\pi_N(O)-\mathbb E_{\nu^O}\Bigl(\int_{[0,1]}\psi \, d\pi_N(O)\Big|\Bigr)>\epsilon\Bigr)\le \frac 1{N^2\epsilon^2} \sum_{k,\ell =1}^N C_{k,\ell}\psi(k/N)\psi(\ell/N).
\end{equation}
Using the bounds of Proposition \ref{prop2.7} we have finally that the left hand side of  \eqref{4.28a} is bounded by
\begin{align}
&\frac{(T_+-T_-)^2}{N+1}\frac{2}{\epsilon^2N^2}\sum_{k<\ell}\frac kN\Bigl(1-\frac{\ell}{N}\Bigr)\psi(k/N)\psi(\ell/N)\nonumber \\
&+\frac{(T_+-T_-)^2}{N+1}\frac{1}{\epsilon^2N^2}\sum_{k=1}^N\frac kN\Bigl(1-\frac{k}{N}\Bigr)\psi^2(k/N)\nonumber \\
& +\frac{(T_+-T_-)^2}{2N(N+1)}\frac{1}{\epsilon^2N^2}\sum_{k=1}^N\psi^2(k/N)\,,
\end{align}
which converges to zero as $N\to\infty$.
\qed 

Adapting the arguments in \cite{MR0630334,KMP} to the absorbed discrete KMP model that we have (that differs just by the boundary dynamics), it is possible to obtain stronger convergence results like local equilibrium, meaning  $\lim_{N\to +\infty}\nu^O(\tau_{uN}f)=\mu^{m(u)}(f)$, where $\mu^{m(u)}$ is a homogeneous product measure of exponential distributions of mean $m(u)$, and the convergence holds for any cylindrical function $f$.

\subsection{Almost sure Duality}
\label{cwalk3}
In this subsection we show that the opinion model is in dual relation with a system of random walks in a space-time random environment. 

The graphical construction of the process is done with a marked Poisson process $(\mathcal N, \mathcal V)$ with the same law as $(\cN,\cU)$. Fix the boundary values $O_{\pVV}(t)\equiv T_{\pVV}$. For $\tau\in \mathcal N_{ij}$, assume $O:=O(\tau-)$ is defined, denote  $V=V(\tau)$, and recalling \eqref{ho1} define 
\begin{equation}
O_\ell(\tau)  :=\bigl(H^O_{ij;V}O(\tau-)\bigr)_{\ell}.
 \label{ovt5}
\end{equation}

We describe a family of backward random walks as a function of a marked Poisson process $(\cN,\cV,\cU')$, where $\cU'$ has the same law as $\cV$ and it is independent of $(\cN,\cV)$. Call $\PP$ and $\EE$ the associated probability and expectation.

Let $s\in\R$, $k\in\oVV$ and let $(R_{k,s}(t))_{t<s}$ be a backwards random walk with initial point $k$ at time $s$:
\begin{align}
    \label{rell1}
  R_{k,s}(s):=k,
\end{align}
and governed by $(\cN,\cV,\cU')$ as follows. Let $\tau\in\cN_{ij}$, denote $V=V(\tau)$, $U'=U'(\tau)$, assume $R_{k,s}(\tau+)=\ell\in\{i,j\} \subset\oVV$, and define
\begin{align}
  R_{k,s}(\tau):=
  \begin{cases}
    i&\ell\in \VV \text{ and } U'<V\\
    j& \ell\in \VV \text{ and } U'>V\\
    j& \ell=j \text{ and } j\in \pVV.
  \end{cases}
       \label{rell2}
\end{align}
The third option in \eqref{rell2} implies that the walk is absorbed at the boundary vertices. When
$R_{k,s}(\tau+)\not\in\{i,j\}$ then $R_{k,s}(\tau)=R_{k,s}(\tau+)$.

Then, analogously (but backwards) to \eqref{xttt}, define
\begin{align}
  \label{rell3}
  R_{k,s}(t) &:= R_{k,s}(\tau_t),\\
  \tau_t&:= \inf\{\tau\in\cN\cap[t,\infty]\}. \label{tau2}
\end{align}

The process $R_{k,s}(t)$ is just a random walk that at rate $1/2$ jumps to a neighbor. But when one fixes $(\cN,\cV)$, it is a ``random walk in a space-time random environment'', as follows.  Denoting $V=V(\tau)$, if $R_{k,s}(\tau+)\in\{i,j\}$ for  $\tau\in\cN_{ij}$, then  at time $\tau$ the walk will be at $i$ with probability $V$, or at $j$ with probability $(1-V)$. Notice that if two walks are respectively at $i$ and $j$, then they coalesce at time $\tau$.
\\
\begin{proposition}[Almost sure duality] Fix a time $s>0$. Let $O(s)$ be constructed as a function of $(\cN,\cV)$ as in \eqref{ovt5}, and with initial configuration $O(0)$. Let the family of processes $\bigl(R_{k,s}(\cdot): k\in\oVV\bigr)$ be constructed with  $(\cN,\cV,\cU')$ as in \eqref{rell3}. Then the following almost sure duality relation holds:
  \begin{align}
    \label{dualr}
      O_{k}(s) = \EE \bigl(O_{R_{k,s}(0)}(0)\,\big|\, (\cN,\cV)\bigr),\qquad \PP\text{-a.s.},\quad k\in\oVV.
  \end{align}
\end{proposition}
\begin{proof} Let $\tau\in\cN_{ij}\cap[0,s]$ and denote $V=V(\tau)$. If $j\in\pVV$ we have
     \begin{align}
     O_{j}(\tau)=O_{j}(\tau-)= \EE\bigl( O_{R_{k,s}(\tau)}(\tau)\,\big|\,R_{k,s}(\tau+)=j\,,\,(\cN,\cV)\bigr).
  \end{align}
Otherwise, $j\in\VV$ and we have
   \begin{align}
     O_{i}(\tau)= O_{j}(\tau) &= V O_i(\tau-) + (1-V) O_j(\tau -)\\
     &= \EE\bigl( O_{R_{k,s}(\tau)}(\tau)\,\big|\,R_{k,s}(\tau+)\in\{i,j\}\,,\,(\cN,\cV)\bigr).
   \end{align}
Use the strong Markov property to conclude.
\end{proof}

We now describe the invariant measure for the local agreement process in terms of the expected absorption locations of the walks.
The hitting time of the boundary for the walk $R_{k,s}$ is defined by
\begin{align}
  \theta_{k,s} := \sup \{t<s: R_{k,s}(t)\in\pVV\}.
\end{align}
In particular, the hitting vertex $R_{k,s}(\theta_{k,s})$ has opinion $T_{R_{k,s}(\theta_{k,s})}$.
\\
\begin{proposition}[Dual expression for the invariant measure]
  \label{invo}
Let the family of walks \[R=(R_{k,s})_{k\in\oVV,s\in\R}\] be governed by $(\cN,\cV,\cU')$, as defined in \eqref{rell1}-\eqref{rell3}. Define
\begin{align}
  \label{tell21}
  O_k^{\stat}(s) := \EE (T_{R_{k,s}(\theta_{k,s})}\,|\,(\cN,\cV)), \qquad k\in\oVV,\,s\in\R.
\end{align}
Then, the process $O^{\stat}(t)$ is a stationary version of the process with generator~\eqref{genO}. In particular, the law of $O^{\stat}(s)$ does not depend on $s$ and it is the unique invariant measure for the opinion model with boundary conditions $T_{\pVV}$.
\end{proposition}
\begin{proof}
  This is a classic application of duality, \cite{KMP,liggett85,durrett88}.  Take $t<s$,  denote $\tPP:=\PP(\,\cdot\,|(\cN,\cV))$, let $O'$ be an arbitrary opinion configuration, define  $O'(t):=O'$ and
  \begin{align}
    \label{oinv1}
    O_k'(s) &:= \sum_{j\in\VV} \tPP(R_{k,s}(t)=j)\, O'_j
               + \sum_{j\in\pVV} \tPP(R_{k,s}(t)=j)\,T_j.
  \end{align}
By duality the process $ (O_k'(s))_{s\ge t}$ has the law of the opinion model $O(s-t)$ with initial configuration $O'$. Taking $t\to-\infty$, the first term vanishes because all random walks are absorbed at the boundaries. Since $\theta_{k,s}>-\infty$ almost surely for all $k,s$,  we get
  \begin{align}
    \label{oinv2}
 \lim_{t\to-\infty}O'_k(s) &= \sum_{j\in\pVV} \tPP(R_{k,s}(\theta_{k,s})=j)\,T_j =: O_k^{\stat}(s).
  \end{align}
Since the law of the $O^{\stat}(s)$ does not depend on~$s$, its law is the invariant measure $\nu$ with boundary conditions $T_{\pVV}$.
\end{proof}

\subsection{Spiking disagreement edges}
\label{S44}
There is no explicit expression for the invariant measure of the opinion model $O$ so far. However, for a process that encodes the edges with disagreeing individuals (or equivalently the pair of neighboring oscillators with different temperatures) it is possible to construct a perfect simulation algorithm and obtain several properties of the invariant measure.

Given an opinion configuration $O$, we say that $ij\in \oE$ is a disagreement edge if $O_i\neq O_j$, otherwise $ij$ is an agreement edge. The disagreement indicator function for edge $ij\in\oE$ at time $t$ is a function of the opinion configuration $O(t)$ defined by
\begin{align}
  \label{eta2}
  \eta_{ij}(t):= \one\{O_i(t)\neq O_j(t)\},\qquad ij\in\oE.
\end{align}
Assume that the graph has a unique connected component and that the boundary conditions are non constant. In this case there is a unique invariant opinion measure $\nu^O$. The stationary opinion configuration $O^{\stat}$ induces a disagreement configuration $\eta^{\stat}$ via \eqref{eta2} satisfying: (a) boundary edges have value $1$, and (b) internal edges sharing a vertex cannot take the value 0 simultaneously. For initial configurations satisfying (a) and (b) the $\eta$ dynamics is very simple and depends only on $\cN$. When $\tau\in\cN_{ij}$, all internal edges sharing exactly one vertex with $ij$ become disagreement edges, and if $ij$ is internal, it becomes an agreement edge. Boundary edges and edges not sharing vertices with $ij$ do not change. More precisely,
\begin{align}
  \label{etat4}
  \eta_{ij}(\tau) &:= H_{ij}\eta(\tau-)\\
 (H_{ij}\eta)_{k\ell}&:=  \begin{cases}
    0&k\ell=ij\text{ and }j\in\VV\\
    1&|\{k,\ell\}\cap\{i,j\}|=1\\
    \eta_{k\ell}&\text{else}
    \end{cases}\label{etat5}
\end{align}
This implies that
the process is Markov in the set
\begin{align}
  \label{etaspace}
 \cE_n &:=\Bigl\{\eta\in\{0,1\}^{\oE}: \eta(ij)=1 \text{ for all }ij\in\pE,\text{ and }\\
  &\qquad\eta(ij)+\eta(k\ell)>0 \text{ if }\{k\ell,ij\}\subset E\text{ and } |\{i,j\}\cap\{k,\ell\}|=1\Bigr\},
\end{align}
with generator
\begin{align}
  L^\eta f(\eta) &:= \sum_{ij\in \oE} \eta_{ij}\,\bigl[f(H_{ij}\eta)-f(\eta)\bigr],\label{Leta26}
\end{align}
where $|J|$ is the cardinal of the set $J$. Given a graph $(\oV,\oE)$ an independent set of edges is a subset of $\oE$ such that all its elements are disjoint (do not share vertices). Shortly, configurations $\eta$ in $\cE_n$ are such that the set of edges with zero value in $\eta$ is an independent subset of $E$ and all the boundary edges have value one.

The process $\eta(t)$ can be seen as a neuron spiking model: when a neuron located at the edge $ij$ has ``energy'' 1, then it spikes at rate 1; at spiking times all internal  edges sharing a vertex with $ij$, get energy 1, and $ij$ drops its energy to 0;  neurons at boundary edges keep their energy equal to 1 at all times.

\paragraph{\sl Disagreement in the homogeneous case.} 
When the boundary conditions $T_{\pE}\equiv 1$, the unique invariant measure for the $O$ process is the delta function concentrating on the ``all one'' configuration, with corresponding $\eta$ configuration ``all zero''. But if the agreement configuration $\eta$ associated to $O(0)$ is in the set \eqref{etaspace}, $O(t)$ converges to ``all one'' but $\eta(t)$ converges to its unique invariant measure which does not coincide with ``all zero'', because $O(t)$ converges to ``all one'' but never attains that point.

\paragraph{\sl Disagreement in gossip/averaging models.} 
We remark that the gossip/averaging models \cite{zbMATH05659657,zbMATH06050911} have the same disagreement process $\eta(t)$ as our opinion process $O(t)$. This implies that the results in this section extend to those models.

\paragraph{Construction (perfect simulation) of $\mu^\eta$}  
We now show that the process $\eta(t)$ has a unique invariant measure $\mu^\eta$, by perfect simulation.
Consider a uniform random permutation $\sigma:\{1,\dots, |\oE|\}\to \oE $. Let $\eta_{ij}[0]= 0$, for internal $ij$ and $1$ for boundary $ij$. Let $E[0]:=\pE$, and define iteratively
\begin{align}
  \label{etas1}
\eta_{ij}[n+1]&:=  \eta_{ij}[n] + \one\{ ij\notin E[n],\;\text{ and } ij \text{ is neighbor of $\sigma(n+1)$\}}\,,\\[3mm]
     E[n+1]&:= E[n] \cup \{\sigma(n+1) \text{ and its neighbors} \}\,;
\end{align}
$\eta[n]$ is the configuration obtained when the edges $\sigma(1),\dots,\sigma(n)$ and their neighbors have been updated. $E[n]$ is the set of edges in $\eta[n]$ that have been updated.

Call $\eta^*$ the configuration obtained when $E[n]$ hits $\oE$.

\paragraph{\sl Example of simulation of $\eta^*$} Consider $\oE=\{12,23,34,45,56,67,78,89\}$ and $\pE=\{12,89\}$.
The 8 edges are represented by the pairs of successive digits of {\ttfamily 123457689} in the top line of the array below.
In the second line the edge $ij$ is represented by {\ttfamily i}, and {\ttfamily  36581247} is an arbitrary permutation of the eight edges represented by {\ttfamily 12345768}.
The third line $\eta[0]=$ {\ttfamily  1......1} is the initial configuration; dots represent edges not in $E[n]$, in particular, $E[0]=\pE$, the boundary edges with value 1.\\\penalty-1000

{\parskip 0pt
  \obeylines
  \centering
  \ttfamily
  \obeyspaces
\phantom{$\eta[0]$}\hspace{5.5mm}  1 2 3 4 5 6 7 8 9
\phantom{$\eta[0]$}\ \ \   3 6 5 8 1 2 4 7
$\eta[0]$\ \ \   1 .\ .\ .\ .\ .\ .\ 1
$\eta[1]$\ \ \   1 .\ .\ 1 0 1 .\ 1
$\eta[2]$\ \ \   1 .\ .\ 1 0 1 1 1
$\eta[3]$\ \ \   1 1 .\ 1 0 1 1 1
$\eta[4]$\ \ \   1 1 .\ 1 0 1 1 1
$\eta[5]$\ \ \   1 1 0 1 0 1 1 1
}\vspace{3mm}

The simulation works as follows. [1] Put {\ttfamily 101} centered at {\ttfamily 1}. [2] Try to put {\ttfamily 101} centered at {\ttfamily 2}, but there is a unique dot to the right of 2, so put just a {\ttfamily 1} there. [3] Try to put {\ttfamily 101} centered at {\ttfamily 3}, but {\ttfamily 3} is a boundary edge, so just put a {\ttfamily 1} to the right of {\ttfamily 3}. [4] Try to put {\ttfamily 101} centered at {\ttfamily 4}, but all edges are already updated, so do nothing. [5] Try to put {\ttfamily 101}  centered at {\ttfamily 5}, but just put the $0$ as the other sites are already updated. Stop because there are no more dots, meaning  $E[5]=\oE$. The bottom line configuration $\eta[5]$  is $\eta^*$.
\\
\begin{proposition}[Invariant measure for the spiking disagreeing process]
Assume the graph $(\oVV,\oE)$ is connected. Assume $ij\mapsto \sigma(ij)$ is a uniform random permutation of $\oE$, then the distribution of $\eta^*$ is the unique invariant measure for the spiking process $\eta(t)$.
\end{proposition}
\begin{proof}
We use the construction \eqref{etat4} and \eqref{etat5} and observe that when $\tau\in\cN_{ij}$ and $\eta_{ij}(\tau-)=1$ for a internal edge $ij$, then at time $\tau$ the value at $ij$ becomes 0, and the values at the neighbors of $ij$ become 1, disregarding the previous values. On the other hand, if $\eta_{ij}(\tau-)=0$, we know that the neighbors of $ij$ are 1, hence we know that at time $\tau$ the values at $ij$ will be 0 and at the neighbors will be 1, independently of the configuration at time $\tau-$.

We construct $\eta(0)$ exploring the configuration backwards. It is only necessary to look at the biggest negative $\tau$ associated to each $ij$. Since only the order of $\tau$'s matter, we can construct $\eta[1],\dots, \eta^*$ using the permutation induced by the order of the $\tau$'s going back in time.
\end{proof}

\paragraph{Markovian behavior of marginals}
We can compute the density of $\eta_{ij}^{\stat}$. If $\eta_{ij}(\tau-)=1$ and  $\tau\in\cN_{ij}$, then $\eta_{ij}(\tau)=0$. On the other hand, if $\eta_{ij}(\tau-)=0$, then $\eta_{k\ell}(\tau-)=1$ for $k\ell$ sharing a vertex with $ij$, and this configuration is preserved at time $\tau$. Hence $\eta_{ij}(t)$ performs a Markov process on $\{0,1\}$ with generator
\[
  \begin{pmatrix}
   -n&n\\
 1& -1
\end{pmatrix}
\]
where $n$ is the number of edges $k\ell$ sharing exactly one internal vertex with $ij$. The invariant measure for this chain is  $\mu^\eta(\eta_{ij}=0) = \frac1{n+1}$ and $\mu^\eta(\eta_{ij}=1) = \frac{n}{n+1}$.

The same argument works for any subset $E'\subset \oE$ that has a Markovian evolution that is the same as you consider as boundary edges, and then freezed at the value 1, all the edges of $\oE$ that share one vertex with an edge of $E'$.

\paragraph{Brownian web} Assume $\VV$ is a integer segment. Then the $\eta$ process can be thought of as  coalescing branching random walks. It is a continuous analogous to the space-time discrete process of coalescing random walks, which are created at all times introduced by Arratia \cite{zbMATH03782178}. We conjecture that in a diffusive limit, the trajectories of the process $\eta(t)$ converge to the Brownian web in the space$\times$time band $[0,1]\times\R$, with absorbing boundary conditions, see Fontes, Isopi, Newman, Ravishankar \cite{zbMATH02148678} and Toth and Werner \cite{zbMATH01203670}.

\section{Discrete boundary driven KMP process}\label{sec:discrete-KMP}
We consider the boundary driven discrete KMP process with generator \eqref{genK}.
Denote $\uk=(k_i)_{i\in\oVV}$, $\us=(s_i)_{i\in\oVV}$, let $\varphi$ be a test function, and define $\mu^K$ by
  \begin{equation}\label{iz2}
  \int\,\mu^K(d \uk)\,\varphi(\uk) :=\int \nu^O(d\us)
\sum_{\uk}\Bigl( \varphi(\uk)\,\tprod_{i\in\oVV} \bigl(\tfrac{s_i}{s_i+1}\bigr)^{k_i}\tfrac{1}{s_i+1}\,\Bigr)  \,,
\end{equation}
where $\nu^O$ is the unique invariant measure for the Markov opinion process $O(t)$ with generator \eqref{genO}. That is, $\mu^K$ is a product of geometric random variables with random means distributed with $\nu^O$.
\begin{theorem}[Invariant measure for discrete KMP]
  \label{tdkmp}
    The measure $\mu^K$ is invariant for the discrete KMP process with generator \eqref{genK} and boundary conditions~$T_{\pVV}$.
\end{theorem}

Theorem \ref{tdkmp} is proven by coupling continuous and discrete KMP.

A pictorial and suggestive way to couple a configuration $K$ with law $\mu^K$ and a configuration $\zeta$ with law $\mu^\zeta$ using rate one Poisson processes is given by the following lemma, illustrated in Figure \ref{zk0}.
\\
\begin{lemma}[Coupling stationary discrete and continuous KMP]
  \label{lem11}
  Let $\cK=(\cK_i)_{i\in\oVV}$ be independent rate 1 Poisson processes in $\R$, $\zeta$ be distributed with $\mu^\zeta$, and $\cK$ and $\zeta$ independent. Defining $K_i:= |\cK_i\cap (0, \zeta_i)|$, we have that $K=(K_i)_{i\in\oVV}$ is distributed with $\mu^K$.
\end{lemma}
\begin{proof} Let $\uk=\uk(\zeta,\cK)$ defined by $k_i:= |\cK_i\cap (0,\zeta_i)|$. Then,
  \begin{align}
    & \int \mu^\cK(d\cK) \mu^\zeta(d\zeta) \,\varphi\bigl(\uk(\zeta,\cK)\bigr)\notag\\
    & \qquad = \iint  \mu^\cK(d\cK)\,\nu^O(d\us) \Bigl(\tprod_{i\in\oVV} \int \frac1{s_i}
      e^{-\zeta_i/s_i} d \zeta_i \Bigr)\,\varphi(\uk(\zeta,\cK))
      = \int\,\mu^K(d\uk)\,\varphi(\uk),\notag
  \end{align}
  because a Poisson random variable with random exponential parameter has geometric distribution with the same mean as the exponential, an observation already noticed by \cite{KMP}. We have also used independence of $\zeta_i$ given $\us$ and the independence property of the Poisson process.
\end{proof}
{\centering
	\includegraphics[width=.8\textwidth, trim= 0 5cm 0 0, clip]{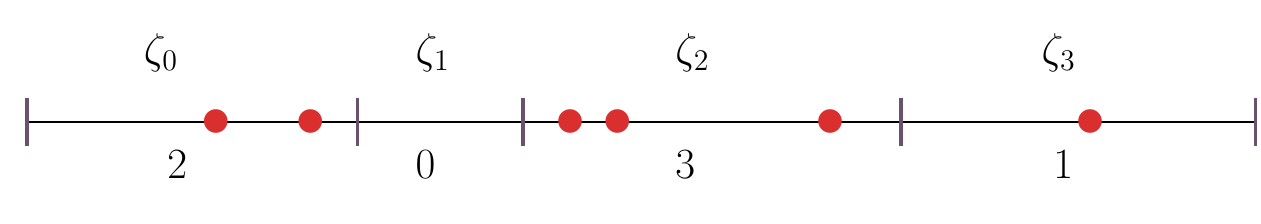}
	\captionof{figure}{Illustration of the coupling $(\zeta,\cK,K)$. The segments of length $\zeta_i$ are displayed successively. Red dots represent the points of the Poisson processes $\cK_i$ in the segments $(0,\zeta_i)$. The number of dots in the segment of length $\zeta_i$ is $K_i$. \label{zk0} }}

\paragraph{Coupling discrete and continuous KMP}

We define a coupling $(\zeta,\kappa,K)(t)=(\zeta_i(t),\kappa_i(t),K_i(t))_{i\in\oVV}$, where the $\zeta$ marginal is the continuous KMP \eqref{zetat1}, and $\kappa_i(t)$ consists of $K_i(t)$ points in $[0,\zeta_i(t)]$, and the marginal $(K_i(t))_{i\in\oVV}$ is the discrete KMP.

The construction of $\zeta_i(t)$ is performed as before. Consider the marked point process $(\cN,\cU,\cB)$, recalling that the marks in $\cU$ are uniform in $[0,1]$ and those in $\cB$ exponential$(1)$ random variables.  Let $\tau\in\cN_{ij}$ with associated $U$ and $B$ (if $j$ is in the boundary), denote $\zeta=\zeta(\tau-)$, and update using the recipe \eqref{ze2}:
\begin{align}
  \zeta_\ell(\tau)
  &:=
\begin{cases}
   U \bigl(\zeta_i + \zeta_{j}\bigr) &\ell=i\\
(1-  U) \bigl(\zeta_i + \zeta_{j}\bigr) & \ell=j\text{ and } j\in \VV\\
     BT_\ell&\ell=j\text{ and } j\in \pVV\\
 \zeta_\ell&\text{otherwise}.%
\end{cases}\label{ze5}
\end{align}
Notice that here we are not using the $(X(t),T(t))$ process to construct $\zeta(t)$.

To construct $\kappa(t)$ we need another two sets of marks. Denote
\begin{align}
  \cK:= \text{Poisson process of intensity 1 in }\R.
\end{align}
Denote $\Gamma= (\Gamma (1),\Gamma (2),\dots)$ a sequence of point processes defined by
\begin{align}
  \Gamma(n) &:= \text{Random set of }n\text{ independent points uniformly distributed in }[0,1].
\end{align}
To each Poisson event $\tau\in\cN_{ij}$ associate a fresh copy of $\Gamma$, and a fresh copy of $\cK$ if $j$ is in the boundary. Call $\cG$ and $\cP$ the respective collection of marks. All together, we get a marked Poisson process $(\cN,\cU,\cB,\cG,\cP)$.

Given a set of points  $Y \subset [0,1]$ and a point $u\in[0,1]$, denote
\begin{align}
 u\star Y &:= Y\cap[0,u]\\
  (1-u)\star Y &:= (Y\cap[u,1])-u ,
\end{align}
these are the set of points in $Y$ that are smaller than $u$, and the set of points in $Y$ bigger than $u$, as seen from $u$.

Let $\tau\in\cN_{ij}$ with associated marks $U,B,\Gamma,\cK$ (where $B$ and $\cK$ appear only if $j$ is in the boundary). Assume we know the configuration of the process until time $\tau-$. Denote  $\zeta= \zeta(\tau-)$, $\kappa=\kappa(\tau-)$ and $K_\ell:=|\kappa_\ell(\tau-)|$. Update $\zeta(\tau)$ using \eqref{ze3}, and
\begin{align}
 \kappa_\ell(\tau)
  &:=
\begin{cases}
  (\zeta_i+\zeta_j)\bigl(U\star\Gamma (K_i+K_j)\bigr) & \ell=i\\
(\zeta_i+\zeta_j)\bigl((1-  U)\star \Gamma (K_i+K_j)\bigr) & \ell=j\text{ and } j\in \VV\\
    \cK\cap[0,BT_\ell)&\ell=j\text{ and } j\in \pVV\\
 \kappa_\ell&\text{otherwise}.%
\end{cases}\label{ze4}
\end{align}
Given the initial condition $(\zeta(0),\kappa(0))$, the process $(\zeta(t),\kappa(t))_{t\ge0}$ governed by the marked Poisson process $(\cN,\cU,\cB,\cP,\cY)$ is defined by \eqref{ze3}, \eqref{ze4} and
\begin{align}
  \label{zk11}
  (\zeta(t),\kappa(t)):=\bigl(\zeta(\tau^t),\kappa(\tau^t)\bigr),
\end{align}
where $\tau^t$ is defined in \eqref{tau1}.

The recipe \eqref{ze4} updates point processes $\kappa_i\subset[0,\zeta_i]$.
When $\tau\in\cN_{ij}$ and $j$ is not in the boundary,  $\zeta_i,\zeta_j$ go to $U(\zeta_i+\zeta_j),(1-U)(\zeta_i+\zeta_j)$; simultaneously, the points in $\kappa_i$ and $\kappa_j$ are remixed together in $[0,\zeta_i+\zeta_j]$, this segment is partitioned into $[0,U(\zeta_i+\zeta_j)]$ and $[U(\zeta_i+\zeta_j),\zeta_i+\zeta_j]$, and the remixed points in those intervals go to the updated $\kappa_i$ and $\kappa_j$, respectively. When $j$ is in the boundary, the $i$ vertex is updated as before (remix before updating, partition the interval and take the point process of the first interval), but $j$ gets a fresh point process, consisting of the Poisson process $\cK$ restricted to the interval $[0,BT_j]$. Consequently, the number of points assigned to the updated $\kappa_j$ has geometric distribution with mean $T_j$. Only the boundary $T_{\pVV}$ show up in this evolution. The next proposition follows immediately from the construction.\\
\begin{proposition}
  Denote $K(t):= |\kappa(t)|$. The process $(\zeta(t),K(t))$ is a coupling between the continuous and discrete KMP processes. More precisely, the marginal processes are Markov with generators $L^\zeta$ and $L^K$, respectively.
\end{proposition}

\paragraph{One dimensional illustration of the coupling}
We consider a small one dimensional lattice with $\VV=\{2,3\}$ and $\pVV=\{1,4\}$.
The evolution of internal vertices is illustrated in Figure~\ref{zk1}. Assume $\tau\in\cN_{23}$, and that the configuration at time $\tau-$ is given in the upper picture. Then, erase the vertical bar separating the segments of sizes $\zeta_2$ and $\zeta_3$ obtaining a segment of size $\zeta_2+\zeta_3$. Resample uniformly the red dots in this segment, obtaining $K_1+K_2$ points in the middle picture. Now choose uniformly a new vertical bar in the segment of size $\zeta_2+\zeta_3$, creating two new segments of sizes $\zeta'_2$ and $\zeta'_3$ and point processes $\cK'_2$, $\cK'_3$, with  $K'_2$ and $K'_3$ points, respectively; the remaining segments are not updated and we obtain the bottom picture at time $\tau$.

{{\centering
	\includegraphics[width=.8\textwidth]{zk1.pdf}
	\captionof{figure}{Illustration of the coupling $(\zeta,\kappa,K)$ in the graph with vertices $\{0,1,2,3\}$, boundary $\{0,3\}$ and nearest neighbor edges.  The $i$-th interval has length $\zeta_i$, and the number below the interval indicates the value of $K_i$. 
          \label{zk1} }}
      \vspace {4mm}

The evolution of the boundary edge is illustrated in Figure \ref{zk3}. Assume $\tau\in\cN_{34}$, where $4$ is a boundary vertex. In the middle picture the bar separating $\zeta_3$ and $\zeta_4$ is erased and the particles of $\kappa_3$ and $\kappa_4$ are remixed. In the bottom line a new separating bar has been added uniformly in the interval of size $\zeta_3+\zeta_4$, $\zeta'_3$, $\kappa'_3$ and $K'_3$ are obtained as in the previous case. The segment of size $\zeta_4$ is substituted by a fresh exponential of mean $T_4$, denoted $\zeta'_4$, the points of $\kappa_4$ are also substituted by $\kappa'_4$, a fresh Poisson Process of rate 1 intersected with the segment of length $\zeta'_4$. Points in $\kappa'_4$ are painted blue. $K'_i$ represent the $K$ configuration after the iteration.

\begin{center}
  	\includegraphics[width=.8\textwidth]{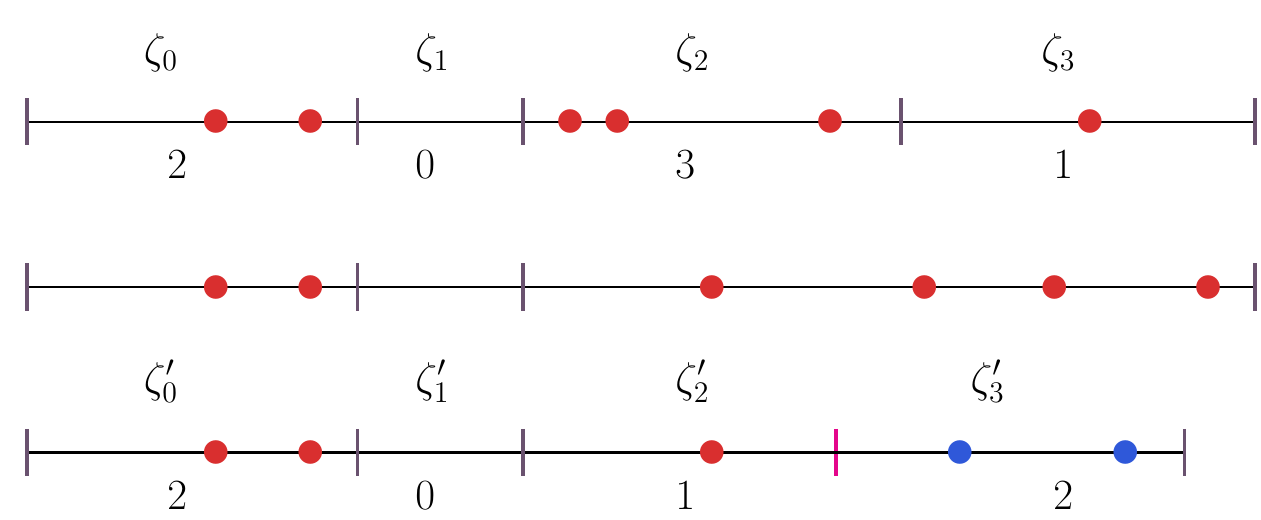}
	\captionof{figure}{Illustration of the coupling $(\zeta,\kappa,K)$ in the same graph as in Fig.\/\ref{zk1} for the boundary edge $23$.
          \label{zk3}}
\end{center}

\paragraph{Invariant measure for the $(\zeta,\kappa)$ process}
Call $\mu^\cP$ the law of $\cP=(\cK_i)_{i\in\oVV}$, iid Poisson processes with parameter $1$ on~$\R$, and recall the invariant measure $\mu^\zeta$ for $\zeta$. Define $\mu^{\zeta,\kappa}(d\zeta,d\kappa)$ as the measure acting on test functions $\varphi:\R_+^{\VV}\times \N_0^{\VV}\to \R$ by
\begin{align}
  \label{zk31}
 \int \mu^{\zeta,\kappa}(d\zeta,d\kappa) \,\varphi(\zeta,\kappa) := \iint \mu^\cK(d\cK)\,\mu^\zeta(d\zeta)\, \varphi\bigl(\zeta,(\cP_i\cap[0,\zeta_i])_{i\in\oVV}\bigr).
\end{align}
Since the product of Poisson processes in intervals are conserved by the operations ``remix and redistribute'' and ``substitute a Poisson random variable by another with the same mean'', and $\mu^\zeta$ is invariant for the $\zeta$ marginal, it is not hard to show the following proposition, which implies Theorem \ref{tdkmp}.
\\
\begin{proposition}[Invariant measure for the $(\zeta,\kappa)$ process]
  \label{izk1}
  The measure $\mu^{\zeta,\kappa}$ defined in \eqref{zk31} is invariant for the coupled process $(\zeta(t),\kappa(t))$ defined by \eqref{ze3},\eqref{ze4} and \eqref{zk11}.
\end{proposition}

\section*{Acknowledgments}
We warmly thank Errico Presutti for constant input and inspiration.

This project started when the first author was visiting Gran Sasso Scientific Institute and Università de L'Aquila in 2022,  and partially developed during the stay of the third author in Universidad de Buenos Aires in the same year. We thank both institutions for warm hospitality and support.

PAF is supported by Consejo Nacional de Investigaciones Científicas y Técnicas, Universidad de Buenos Aires and Funda\c cão de Apoio à Pesquisa do Estado de São Paulo, Neuromat project. 
DG acknowledge financial support from the Italian Research Funding Agency (MIUR) through
PRIN project ``Emergence of condensation-like phenomena in interacting particle systems: kinetic and lattice models'', grant 202277WX43.

\section*{Declarations}
The authors have no conflicts of interest. Data sharing is not applicable to this article as no datasets were generated or analysed during the current study.



\begin{thebibliography}{10}

\bibitem{zbMATH06216071}
David Aldous.
\newblock Interacting particle systems as stochastic social dynamics.
\newblock {\em Bernoulli}, 19(4):1122--1149, 2013.
\newblock \href {https://doi.org/10.3150/12-BEJSP04}
  {\path{doi:10.3150/12-BEJSP04}}.

\bibitem{zbMATH06050911}
David Aldous and Daniel Lanoue.
\newblock A lecture on the averaging process.
\newblock {\em Probab. Surv.}, 9:90--102, 2012.
\newblock \href {https://doi.org/10.1214/11-PS184}
  {\path{doi:10.1214/11-PS184}}.

\bibitem{zbMATH03782178}
Richard Arratia.
\newblock Limiting point processes for rescalings of coalescing and
  annihilating random walks on {{\(Z^ n\)}}.
\newblock {\em Ann. Probab.}, 9:909--936, 1981.
\newblock \href {https://doi.org/10.1214/aop/1176994264}
  {\path{doi:10.1214/aop/1176994264}}.

\bibitem{zbMATH02010952}
E.~Ben-Naim, P.~L. Krapivsky, F.~Vazquez, and S.~Redner.
\newblock Unity and discord in opinion dynamics.
\newblock {\em Physica A}, 330(1-2):99--106, 2003.
\newblock \href {https://doi.org/10.1016/j.physa.2003.08.027}
  {\path{doi:10.1016/j.physa.2003.08.027}}.

\bibitem{MFT}
L.~Bertini, A.~De~Sole, D.~Gabrielli, G.~Jona-Lasinio, and C.~Landim.
\newblock Macroscopic fluctuation theory for stationary nonequilibrium states.
\newblock {\em J. Stat. Phys.}, 107(3-4):635--675, 2002.
\newblock \href {https://doi.org/10.1023/A:1014525911391}
  {\path{doi:10.1023/A:1014525911391}}.

\bibitem{Bertini_2007}
L~Bertini, A~De Sole, D~Gabrielli, G~Jona-Lasinio, and C~Landim.
\newblock Stochastic interacting particle systems out of equilibrium.
\newblock {\em Journal of Statistical Mechanics: Theory and Experiment},
  2007(07):P07014, jul 2007.
\newblock URL: \url{https://dx.doi.org/10.1088/1742-5468/2007/07/P07014}, \href
  {https://doi.org/10.1088/1742-5468/2007/07/P07014}
  {\path{doi:10.1088/1742-5468/2007/07/P07014}}.

\bibitem{BGL}
Lorenzo Bertini, Davide Gabrielli, and Joel~L. Lebowitz.
\newblock Large deviations for a stochastic model of heat flow.
\newblock {\em J. Stat. Phys.}, 121(5-6):843--885, 2005.
\newblock \href {https://doi.org/10.1007/s10955-005-5527-2}
  {\path{doi:10.1007/s10955-005-5527-2}}.

\bibitem{1498447}
S.~Boyd, A.~Ghosh, B.~Prabhakar, and D.~Shah.
\newblock Gossip algorithms: design, analysis and applications.
\newblock In {\em Proceedings IEEE 24th Annual Joint Conference of the IEEE
  Computer and Communications Societies.}, volume~3, pages 1653--1664 vol. 3,
  2005.
\newblock \href {https://doi.org/10.1109/INFCOM.2005.1498447}
  {\path{doi:10.1109/INFCOM.2005.1498447}}.

\bibitem{1638541}
S.~Boyd, A.~Ghosh, B.~Prabhakar, and D.~Shah.
\newblock Randomized gossip algorithms.
\newblock {\em IEEE Transactions on Information Theory}, 52(6):2508--2530,
  2006.
\newblock \href {https://doi.org/10.1109/TIT.2006.874516}
  {\path{doi:10.1109/TIT.2006.874516}}.

\bibitem{zbMATH07733575}
Pietro Caputo, Matteo Quattropani, and Federico Sau.
\newblock Cutoff for the averaging process on the hypercube and complete
  bipartite graphs.
\newblock {\em Electron. J. Probab.}, 28:31, 2023.
\newblock Id/No 100.
\newblock URL:
  \url{projecteuclid.org/journals/electronic-journal-of-probability/volume-28/issue-none/Cutoff-for-the-averaging-process-on-the-hypercube-and-complete/10.1214/23-EJP993.full},
  \href {https://doi.org/10.1214/23-EJP993} {\path{doi:10.1214/23-EJP993}}.

\bibitem{modena}
G.~Carinci, C.~Franceschini, R.~Frassek, C.~Giardinà, and F.~Redig.
\newblock The open harmonic process: non-equilibrium steady state, pressure,
  density large deviation and additivity principle.
\newblock {\em Preprint, ArXiv 2307.14975}, 2023.

\bibitem{Har}
G.~Carinci, C.~Franceschini, D.~Gabrielli, C.~Giardin{\`a}, and
  D.~Tsagkarogiannis.
\newblock Solvable stationary non equilibrium states.
\newblock {\em J. Stat. Phys.}, 191(1):14, 2024.
\newblock Id/No 10.
\newblock \href {https://doi.org/10.1007/s10955-023-03226-z}
  {\path{doi:10.1007/s10955-023-03226-z}}.

\bibitem{castellano2009}
Claudio Castellano, Santo Fortunato, and Vittorio Loreto.
\newblock Statistical physics of social dynamics.
\newblock {\em Reviews of modern physics}, 81(2):591, 2009.

\bibitem{db-2022}
Thomas Dag{\`e}s and Alfred~M. Bruckstein.
\newblock Doubly stochastic pairwise interactions for agreement and alignment.
\newblock {\em SIAM J. Appl. Math.}, 82(4):1246--1266, 2022.
\newblock \href {https://doi.org/10.1137/21M1394680}
  {\path{doi:10.1137/21M1394680}}.

\bibitem{deffuant2000mixing}
Guillaume Deffuant, David Neau, Frederic Amblard, and G{\'e}rard Weisbuch.
\newblock Mixing beliefs among interacting agents.
\newblock {\em Advances in Complex Systems}, 3(01n04):87--98, 2000.

\bibitem{DEL}
B.~Derrida, C.~Enaud, and J.~L. Lebowitz.
\newblock The asymmetric exclusion process and {Brownian} excursions.
\newblock {\em J. Stat. Phys.}, 115(1-2):365--382, 2004.
\newblock \href {https://doi.org/10.1023/B:JOSS.0000019833.35328.b4}
  {\path{doi:10.1023/B:JOSS.0000019833.35328.b4}}.

\bibitem{MR2335699}
Bernard Derrida.
\newblock Non-equilibrium steady states: fluctuations and large deviations of
  the density and of the current.
\newblock {\em J. Stat. Mech. Theory Exp.}, (7):P07023, 45, 2007.
\newblock \href {https://doi.org/10.1088/1742-5468/2007/07/p07023}
  {\path{doi:10.1088/1742-5468/2007/07/p07023}}.

\bibitem{durrett88}
Richard Durrett.
\newblock {\em Lecture notes on particle systems and percolation}.
\newblock Pacific Grove, CA: Wadsworth \&| Brooks/Cole Advanced Books \&|
  Software, 1988.

\bibitem{EnauD}
C.~Enaud and B.~Derrida.
\newblock Large deviation functional of the weakly asymmetric exclusion
  process.
\newblock {\em J. Stat. Phys.}, 114(3-4):537--562, 2004.
\newblock \href {https://doi.org/10.1023/B:JOSS.0000012501.43746.cf}
  {\path{doi:10.1023/B:JOSS.0000012501.43746.cf}}.

\bibitem{fgc2023}
Simone Floreani and Adrián~González Casanova.
\newblock Non-equilibrium steady state of the symmetric exclusion process with
  reservoirs.
\newblock {\em ArXiv 2307.02481}, 2023.

\bibitem{zbMATH02148678}
L.~R.~G. Fontes, M.~Isopi, C.~M. Newman, and K.~Ravishankar.
\newblock The {Brownian} web: characterization and convergence.
\newblock {\em Ann. Probab.}, 32(4):2857--2883, 2004.
\newblock \href {https://doi.org/10.1214/009117904000000568}
  {\path{doi:10.1214/009117904000000568}}.

\bibitem{MR0630334}
A.~Galves, C.~Kipnis, C.~Marchioro, and E.~Presutti.
\newblock Nonequilibrium measures which exhibit a temperature gradient: study
  of a model.
\newblock {\em Comm. Math. Phys.}, 81(1):127--147, 1981.
\newblock URL: \url{http://projecteuclid.org/euclid.cmp/1103920162}.

\bibitem{ghh2020}
Nina Gantert, Markus Heydenreich, and Timo Hirscher.
\newblock {Strictly weak consensus in the uniform compass model on
  $\mathbb{Z}$}.
\newblock {\em Bernoulli}, 26(2):1269 -- 1293, 2020.
\newblock \href {https://doi.org/10.3150/19-BEJ1155}
  {\path{doi:10.3150/19-BEJ1155}}.

\bibitem{zbMATH06092639}
Javier G{\'o}mez-Serrano, Carl Graham, and Jean-Yves Le~Boudec.
\newblock The bounded confidence model of opinion dynamics.
\newblock {\em Math. Models Methods Appl. Sci.}, 22(2):1150007, 46, 2012.
\newblock \href {https://doi.org/10.1142/S0218202511500072}
  {\path{doi:10.1142/S0218202511500072}}.

\bibitem{hagsstrom-2012}
Olle H{\"a}ggstr{\"o}m.
\newblock A pairwise averaging procedure with application to consensus
  formation in the {Deffuant} model.
\newblock {\em Acta Appl. Math.}, 119(1):185--201, 2012.
\newblock \href {https://doi.org/10.1007/s10440-011-9668-9}
  {\path{doi:10.1007/s10440-011-9668-9}}.

\bibitem{Hirscher-2017}
Timo Hirscher.
\newblock Overly determined agents prevent consensus in a generalized
  {Deffuant} model on {{\(\mathbb{Z}\)}} with dispersed opinions.
\newblock {\em Adv. Appl. Probab.}, 49(3):722--744, 2017.
\newblock \href {https://doi.org/10.1017/apr.2017.19}
  {\path{doi:10.1017/apr.2017.19}}.

\bibitem{KMP}
C~Kipnis, C~Marchioro, and E~Presutti.
\newblock Heat flow in an exactly solvable model.
\newblock {\em Journal of Statistical Physics}, 27(1):65--74, 1982.

\bibitem{zbMATH06235591}
Nicolas Lanchier.
\newblock The critical value of the {Deffuant} model equals one half.
\newblock {\em ALEA, Lat. Am. J. Probab. Math. Stat.}, 9(2):383--402, 2012.
\newblock URL: \url{alea.impa.br/articles/v9/09-16.pdf}.

\bibitem{lanchier-li-2020}
Nicolas Lanchier and Hsin-Lun Li.
\newblock Probability of consensus in the multivariate {Deffuant} model on
  finite connected graphs.
\newblock {\em Electron. Commun. Probab.}, 25:12, 2020.
\newblock Id/No 79.
\newblock \href {https://doi.org/10.1214/20-ECP359}
  {\path{doi:10.1214/20-ECP359}}.

\bibitem{liggett85}
Thomas~M. Liggett.
\newblock {\em Interacting particle systems. {With} a new postface.}
\newblock Class. Math. New York, NY: Springer, reprint of the 1985 edition
  edition, 2005.

\bibitem{picci-taylor2013}
Giorgio Picci and Thomas~J. Taylor.
\newblock Almost sure exponential convergence to consensus of random gossip
  algorithms.
\newblock {\em Int. J. Robust Nonlinear Control}, 23(9):1033--1045, 2013.
\newblock \href {https://doi.org/10.1002/rnc.2844}
  {\path{doi:10.1002/rnc.2844}}.

\bibitem{zbMATH07692286}
Matteo Quattropani and Federico Sau.
\newblock Mixing of the averaging process and its discrete dual on
  finite-dimensional geometries.
\newblock {\em Ann. Appl. Probab.}, 33(2):1136--1171, 2023.
\newblock \href {https://doi.org/10.1214/22-AAP1838}
  {\path{doi:10.1214/22-AAP1838}}.

\bibitem{sau2024tiny}
Federico Sau.
\newblock Tiny fluctuations of the averaging process around its degenerate
  steady state, 2024.
\newblock \href {http://arxiv.org/abs/2403.02032} {\path{arXiv:2403.02032}}.

\bibitem{zbMATH05659657}
Devavrat Shah.
\newblock Gossip algorithms.
\newblock {\em Found. Trends Netw.}, 3(1):1--125, 2008.
\newblock URL:
  \url{semanticscholar.org/paper/55e09948fe85ff0d07db5d6ac962e52d5fd10fb3},
  \href {https://doi.org/10.1561/1300000014} {\path{doi:10.1561/1300000014}}.

\bibitem{zbMATH01203670}
B{\'a}lint T{\'o}th and Wendelin Werner.
\newblock The true self-repelling motion.
\newblock {\em Probab. Theory Relat. Fields}, 111(3):375--452, 1998.
\newblock \href {https://doi.org/10.1007/s004400050172}
  {\path{doi:10.1007/s004400050172}}.

\bibitem{vazquez-krapivsky-redner}
F.~Vazquez, P.~L. Krapivsky, and S.~Redner.
\newblock Constrained opinion dynamics: freezing and slow evolution.
\newblock {\em J. Phys. A, Math. Gen.}, 36(3):l61--l68, 2003.
\newblock \href {https://doi.org/10.1088/0305-4470/36/3/103}
  {\path{doi:10.1088/0305-4470/36/3/103}}.

\bibitem{weisbuch2005persuasion}
G{\'e}rard Weisbuch, Guillaume Deffuant, and Fr{\'e}d{\'e}ric Amblard.
\newblock Persuasion dynamics.
\newblock {\em Physica A: Statistical Mechanics and its Applications},
  353:555--575, 2005.

\end{thebibliography}

\end{document}